\documentclass{article}

\usepackage{amsmath,amssymb,amsfonts,amsthm,mathrsfs,paralist}
\usepackage{mathtools,enumitem}
\usepackage{tikz}
\usepackage{setspace}
\usepackage{authblk}
\usepackage{dsfont}
\usepackage{caption,subcaption}
\usepackage{fdsymbol}
\usepackage{tcolorbox}
\usepackage{biblatex} 

\theoremstyle{plain}
\newtheorem{thm}{Theorem}[section]
\newtheorem{prop}[thm]{Proposition}
\newtheorem{lemma}[thm]{Lemma}
\newtheorem{cor}[thm]{Corollary}
\newtheorem{conjecture}[thm]{Conjecture}

\theoremstyle{definition}
\newtheorem{defi}[thm]{Definition}
\newtheorem{notation}[thm]{Notation}
\newtheorem{ex}[thm]{Example}

\theoremstyle{remark}
\newtheorem{remark}[thm]{Remark}

\newcommand{\R}{\mathbb {R}}
\newcommand{\Q}{\mathbb {Q}}
\newcommand{\N}{\mathbb {N}}
\newcommand{\Z}{\mathbb {Z}}

\makeatletter
\newcommand{\vast}{\bBigg@{4}}
\newcommand{\Vast}{\bBigg@{5}}

\makeatother

\definecolor{Blue}{rgb}{0,0,0.7}

\def\iddots{\mathinner{\mkern1mu\raise1pt
    \hbox{.}\mkern2mu\raise4pt\hbox{.}\mkern2mu
        \raise7pt\vbox{\kern7pt\hbox{.}}\mkern1mu}}
\def\Vol{{\rm Vol}}

\DeclareMathOperator{\vol}{Vol}

\DeclareMathOperator{\diag}{diag}
\DeclareMathOperator{\disp}{disp}
\DeclareMathOperator{\disc}{disc}

\DeclareMathOperator{\GL}{GL}

\def\jaspar#1{\fbox {\footnote {\ }}\ \footnotetext { From Jaspar: #1}}

\def\hjaspar#1{}
\def\hthomas#1 {}

\begin{document}

\title{On the area of empty axis-parallel rectangles amidst 2-dimensional lattice points}

\author{Thomas Lachmann and Jaspar Wiart\thanks{The authors are supported by the Austrian Science Fund (FWF),
Project F5505-N26 and F5509-N26, which are part of the Special Research Program ``Quasi-Monte Carlo Methods: Theory and Applications''.}}

% \author{Jaspar Wiart}

% \author{Ralph Kritzinger and Jaspar Wiart\thanks{The authors are supported by the Austrian Science Fund (FWF),
% Project F5509-N26, which is part of the Special Research Program ``Quasi-Monte Carlo Methods: Theory and Applications''.}}

% \thanks{The authors are supported by the Austrian Science Fund (FWF),
% Project F5505-N26 and F5509-N26, which are part of the Special Research Program ``Quasi-Monte Carlo Methods: Theory and Applications''.}

\affil{Johannes Kepler University \authorcr Altenbergerstr.\ 69 \authorcr
4040 Linz, Austria \authorcr thomas.lachmann@jku.at\\jaspar.wiart@jku.at}

%  \affil[2]{Johannes Kepler University \authorcr Altenbergerstr.\ 69 \authorcr
% 4040 Linz, Austria \authorcr  jaspar.wiart@jku.at}
 \date{}
\maketitle

\begin{abstract}
The dispersion of a point set in the unit square is defined to be the area of the largest empty axis-parallel box. In this paper we are interested in the dispersion of lattices in the plane, that is, the supremum of the area of the empty axis-parallel boxes amidst the lattice points. We introduce a framework with which to study this based on the continued fractions expansions of the generators of the lattice. This framework proves so successful that we were unable to ask a question that we could not answer. We give necessary and sufficient conditions under which a lattice has finite dispersion. We obtain an exact formula for the dispersion of the lattices associated to subgroups of the ring of integer of a quadratic field.  We have tight bounds for the dispersion of a lattice based the largest continued fraction coefficient of the generators, accurate to within one half. We know what the $n$-th best lattice is. We provide an equivalent formulation of Zaremba's conjecture. Using our framework we are able to give alternative proofs of the results from two other papers in only a few lines. 
%This paper answers every question you ever had about the area of the largest empty axis-parallel box amidst 2-dimensional lattice points, also known as dispersion. A necessary and sufficient condition for a lattice to have finite dispersion? Got it! An exact formula for the dispersion of a wide class of lattices? Got that too! Sharp bounds to within 1/2 for all other lattices? Believe it! A formula for the $n$-th best lattice? Does a duck like water?
\end{abstract}

%%%%%%%%%%%%%%%%%%%%%%%%%%%%%%%%%%%%%%%%%%%%%%%%%%%%%%%%%%%%%%%%%%%%%%%%%%%

\section{Introduction and preliminaries}

Let $P\subset[0,1]^s$ be a point set of size $n$. The dispersion of $P$, denoted $\disp(P)$, is the volume of the largest empty axis-parallel (open) box amidst the points of $P$. Of interest is the dispersion of the optimal point set of size $n$ in $[0,1]^s$ which we denote by
\[
\disp(n,s)=\min\{\disp(P_n):P_n\subseteq[0,1]^s,\, |P_n|=n\}.
\]
The question of finding optimal dispersion of point sets have been studied for a long while, including allowing for shapes other than axis-parallel boxes, such as in \cite{BreneisHinrichs,Fiat1989HowTF,KRIEG2018115,kritzinger2020dispersion,niederreiter1984measure,ULLRICH2019385}.
It is known that $\disp(n,s)= \Omega(s/n)$ and $\disp(n,s)= O(s^2\log(s)/n)$, with the latter being attained by a modified Hilton-Hammerseley construction in \cite{Bukh}. There they also conjecture that the actual order of $\disp(n,s)$ is in $\Theta(s\log(s)/n)$. We will be restricting ourselves to 2-dimensional dispersion. Because the minimal dispersion is of order $n^{-1}$ for all $s$, the asymptotic behaviour of the minimal dispersion normalized by $n$ is interesting.   In two dimensions it is known that
\[
1.50476\leq \liminf_{n\to\infty}n\disp(n,2) \leq1.89442...
\]
The upper bound (which is conjectured to be optimal) is attained by a modified Fibonacci lattice in \cite{Wiart} and the lower bound was proved in \cite{Bukh}. 

Although the majority of this paper will focus on dispersion, we will spend some time with periodic dispersion, denoted $\disp_{\mathbb T}(P)$. The difference from normal dispersion is that we allow boxes to wrap around, e.g.\ if $x_1>y_1$ and $x_2<y_2$, then $(x_1,y_1)\times(x_2,y_2)$ denotes the set $\big((0,x_1)\cup(y_1,1)\big)\times(x_2,y_2)$. In two dimensions it is known that $\disp_{\mathbb T}(n,2)\geq 2/n$. It was shown in \cite{BreneisHinrichs} that the only rank-1 lattices (see Section \ref{sec:Periodic Disperson}) that have dispersion $2/n$ are Fibonacci lattices (see \eqref{eq:Fibonacci Lattice}). 

Let $\Lambda\subset \R^s$ be a lattice generated by the matrix $A$. The dispersion of $\Lambda$, denoted $\disp(\Lambda)$ is the supremum over the volumes of all empty axis-parallel boxes amidst the points of $\Lambda$. Here we note that the boxes need not be contained in the unit hypercube and that there may not always be a box whose volume equals the dispersion of $\Lambda$. We may obtain a point set $P\subset[0,1]^s$ with roughly $n$ points from a lattice $\Lambda$ by setting 
\[
P=\bigg(\Big(n^{-1}\det(\Lambda)\Big)^{1/s}\Lambda\bigg)\cap[0,1]^s,
\]
where $\det(\Lambda)=\det(A)$.  This assignment gives
$|P|\disp(P)=\disp(\Lambda)/\det(\Lambda)
% |P|\disp(P)=\frac{\disp(\Lambda)}{\det(\Lambda)}
$
plus a small error term depending only on $n$ converging to $0$ as $n$ approaches infinity. In two dimensions, point sets attained in this way from the lattice generated by the matrix
\[
\begin{pmatrix}
1   &   \varphi\\
1   &   \overline\varphi
\end{pmatrix},
\]
with $\varphi$ equal to the golden ratio and $\overline\varphi=1-\varphi$, match the best known asymptotic construction in two dimensions. Moreover, if $F_m$ denotes the $m$-th Fibonacci number, then the intersection of the lattice generated by the matrix
\[
\begin{pmatrix}
F_{m-2}/F_m   &   1\\
1/F_m   &  0
\end{pmatrix}
\]
with the unit square is the Fibonacci lattice of size $F_m$. With the success of lattices in two dimensions, it seems plausible to the authors that the optimal asymptotic point sets in all dimensions will be of this form. By examining two dimensional lattices we hope to gain insight into the higher dimensional case.

Our work is also relevant to geometry of numbers. The celebrated theorem of Minkowski says that every convex and centrally symmetric region with volume strictly greater than $2^d\det(\Lambda)$, must contain a point from the lattice besides the origin.  Dispersion relaxes the condition that a region be central and adds a restriction on the shape. Here it is to be noted that the dispersion can also be defined over any kind of object, usually to other symmetric objects like balls or hypercubes. 

In Section 2 we introduce the notion of normalized box area (NBA) equivalence for lattices, which in particular preserves normalized dispersion, and show that every NBA equivalence class contains a lattice with generating matrix of the form
\[
\begin{pmatrix}
1 & -\Delta\\
1 & -\overline\Delta
\end{pmatrix},
\]
with $\Delta>1$ and $-1<\widetilde{\Delta}<0$ and $B_0=(-\Delta,1)\times(0,1-\widetilde{\Delta})$ being a maximal empty box amidst $\Lambda$. Throughout most of the paper we will assume that both generators are irrational, this guarantees that every empty axis-parallel box is contained in a maximal one.  This assumption is mainly for convenience as statement of our results would otherwise be more complicated.

Section 3 develops the framework used in the rest of the paper. We begin with Proposition \ref{prop:BoxesAreSemiConvergents}, which explores the connection between the continued fraction expansions $\Delta=[a_0,a_1,a_2,\dots]$ and $-\widetilde\Delta=[0,a_{-1},a_{-2},\dots]$, with $a_i\in\N$, and the maximal empty axis-parallel boxes bounded from below by the origin amidst the associated lattice.
We go on to show that each class of dispersion equivalent lattices can be uniquely represented by a two-sided sequence $(a_i)_{i\in\Z}$, up to re-indexing or transposing, obtained from $\Delta$ and $\widetilde\Delta$. We finish the section by proving that a lattice has finite dispersion if and only if its associated sequence is bounded.

In Section 4 we turn our attention to quadratic lattices, i.e.\  two-dimensional embeddings $\Lambda$ of subrings $R$ of the ring of integers of a quadratic field $\Q [\sqrt{d}]$. In Theorem \ref{thm:MainTheorem} we obtain an exact formula for the dispersion of $\Lambda$ in terms of the discriminant of $R$ which shows a linear dependence in $\disc(R)$. The proof of this result requires Theorem \ref{thm:EstimateOnAi}, a technical result concerning the size of the continued fraction coefficients of purely periodic quadratic integers. (We prove this result in Appendix \ref{sec:Proof of EstimateOnAi}.)  In Theorem \ref{thm:Thight Bound} we give tight bounds on the normalized dispersion of a lattice in terms of the maximum of its associated sequence.  The upper and lower bounds for the dispersion of a lattice whose sequence has a given maximum, which differ by less than one half, are attained by certain quadratic lattices.

In Section 5 we find the lattices with the $n$-th smallest normalized dispersion whose sequence converges to 
$
% \frac{4+5^{1/2}}{3}
(4+5^{1/2})/3=2.07868\dots \footnote{Curiously, this is an equation containing the first five natural numbers exactly once.}
$
The generators of the first and second best lattices are both quadratic lattices, while the rest are generated by certain purely periodic quadratic numbers. (It is here that we see that the asymptotic bound found in \cite{Wiart} corresponds to the dispersion of the optimal lattice.) Interestingly, the next best lattice (generated by $1+\sqrt{2}$ and $1-\sqrt{2}$) has normalized dispersion greater than 2. This might help explain why the optimization algorithm in the upcoming dispersion package, by Benjamin Sommer, struggles to optimize a random point set to one with a normalized dispersion smaller than $2$.  Also, the generators of the $n$-th best lattices are Lagrange Numbers corresponding to certain Markov triples (see Remark \ref{rem:Lagrange Connection}), this seems to imply a connection between dispersion and the Lagrange and Markov spectrum which the authors did not further investigate.

Finally, in Section 6 we explain how to modify the framework developed in Section \ref{sec:contfracconnect} in order to drop the assumption that the generators $\Delta$ and $\widetilde\Delta$ are irrational. This allows us to study integration lattices (also known as rank-1 lattices), a common class of low discrepancy point sets. In Theorem \ref{thm:Zaremba eq disp Bound} we demonstrate that Zaremba's Conjecture \ref{con:Zaremba} is equivalent to the existence of a constant $C$ such that for all $n$, there exists an integration lattice with periodic dispersion less than $C/n$. Additionally, by applying our framework, we are able to give a very short proof of the results from \cite{BreneisHinrichs}, i.e.\ that the only integration lattices which achieve a normalized periodic dispersion of 2 are Fibonacci lattices.

\section{Maximal empty boxes amidst lattice points}\label{sec:Empty boxes}

We begin this section by defining an equivalence relation on matrices called normalized box area (NBA) equivalence. NBA equivalent matrices generate lattices with, for our purposes, essentially the same empty box structure. Obviously, if two matrices generate the same lattice they should be equivalent, i.e. multiplication with a matrix in $\GL_2(\Z)$ from the right. Next, since we are interested in the normalized dispersion, and since multiplication by
\[
\begin{pmatrix}
    a  &0\\
    0  &b
\end{pmatrix},
\]
with $a,b>0$, from left does not change the relative position of points of a lattice we also allow this kind of operation.
We also allow transformations by symmetries of a square, as these map axis-parallel boxes to axis-parallel boxes. We summarize these in the following definition.
\begin{defi}
Two matrices are NBA equivalent if they are obtained via the following operations from each other.
Either by multiplication of a generating matrix from the left by diagonal matrices $\diag(a,b)$ where $ab\neq 0$ and the matrix
    \[
    \begin{pmatrix}
    0  &1\\
    1  &0
    \end{pmatrix},
    \]
or from the right by a matrix in $\GL_2(\Z)$.
\end{defi}
The following notation is inspired by the standard notation for the conjugate of an element in a quadratic field. (see Section \ref{sec:dispofquadlatt})
\begin{notation}
We will use $\alpha$, $\beta$, or $\gamma$ to denote the first coordinate elements in a lattice $\Lambda$ and $\widetilde\alpha$, $\widetilde\beta$, or $\widetilde\gamma$ to denote their second coordinates.
\end{notation} 

\begin{lemma}\label{lem:NBAequivoperations}
Two NBA equivalent matrices generate lattices with the same normalized dispersion.
\end{lemma}
\begin{proof}
Let $ab\neq 0$ and consider the transformation of $\mathbb{R}^2$ given by left multiplication by $\diag(a,b)$. Under this transformation an axis-parallel box $B=(x_1,x_2)\times (y_1,y_2) \subset \mathbb{R}^2$ maps to the axis-parallel box
\[
B'=\diag(a,b)B= (ax_1,ax_2)\times (by_1,by_2),
\]
while the lattice $\Lambda$ generated by the matrix $A$ maps to lattice $\Lambda'$ generated by the matrix $\diag(a,b)A$.
Since this transformation is one-to-one the number of points from $\Lambda$ that are in $B$ is the same number of points from $\Lambda'$ that are in $B'$. This together with the fact that the normalized volume is preserved, i.e.\ 
\[
\frac{\Vol(B)}{\det(\Lambda)}=\frac{ab\Vol(B)}{ab\det(\Lambda)}=\frac{\Vol(B')}{\det(\Lambda')},
\]
shows that this transformation preserves normalized dispersion.
For the matrix
\[
    \begin{pmatrix}
    0  &1\\
    1  &0
    \end{pmatrix}
\]
the assertion is clearly true as the transformation is isometric, maps axis-parallel boxes to axis-parallel boxes, and is one-to-one. 
Multiplying $A$ from the right with a matrix from $\GL_2(\Z)$ leaves the lattice invariant.
\end{proof}

\begin{defi}
An \emph{irrational lattice} is one in which no two points appear on the same horizontal or vertical line. We call a matrix that generates an irrational lattice \emph{irrational}.
\end{defi}

The main objective for the remainder of this section is to prove that every irrational matrix is NBA equivalent to a matrix of the form
\[
\begin{pmatrix}
1   &-\Delta\\
1   &-\widetilde\Delta
\end{pmatrix}, 
\]
where $1<\Delta$ and $-1<\widetilde\Delta< 0$.\footnote{The choice of negative signs in the matrix will be cleared up in the beginning of section \ref{sec:contfracconnect} and the restrictions on $\Delta$ and $\widetilde\Delta$ in the beginning of section \ref{sec:dispofquadlatt}.} (It is clear that  $\Delta$ and $\widetilde\Delta$ must be irrational.)
The upcoming lemma shows that we can obtain a generating matrix from the bounding points of a single maximal axis-parallel empty box whose bottom touches the origin.

\begin{lemma}\label{lem:TopPointIsDetermined}
Let $\Lambda$ be a lattice and let $B=(\alpha,\beta)\times(0,\widetilde\gamma)$ be a maximal empty axis-parallel box bounded on the bottom by $(0,0)$, on the left by $(\alpha,\widetilde\alpha)$, on the right by $(\beta,\widetilde\beta)$, and on top by $(\gamma,\widetilde\gamma)$. Then $(\gamma,\widetilde\gamma)=(\alpha,\widetilde\alpha)+(\beta,\widetilde\beta)$ and $\Lambda$ is generated by the matrix
\[
A=\begin{pmatrix}
\beta           &   \alpha\\
\widetilde\beta  &   \widetilde\alpha
\end{pmatrix}.
\]
\end{lemma}

\begin{proof}
Applying the transformation matrix $-I$ to $B$ and then translating by $(\gamma,\widetilde\gamma)$ yields the box $B'=(\gamma-\beta,\gamma-\alpha)\times(0,\widetilde\gamma)$ which is bounded on the bottom by $(0,0)$ and on the top by $(\gamma,\widetilde\gamma)$ and is a maximal empty box because the transformations we applied to $B$ leave the lattice invariant.  Since a two dimensional maximal empty axis-parallel box is completely determined by any two points on opposing sides, we must have $B=B'$. It follows that 
\[
(\alpha,\widetilde\alpha)=(\gamma,\widetilde\gamma)-(\beta,\widetilde\beta).
\]
Since $B\cap\Lambda$ is empty, the only point from $\Lambda$ in the parallelepiped $P=\{Ax:x\in[0,1)\}$ generated by $A$  is $(0,0)$. This, together with the fact that $A$ has full rank, shows that $A$ generates $\Lambda$.
\end{proof}

To find the dispersion of a lattice $\Lambda$ we need only consider boxes maximal empty boxes $B$ amidst $\Lambda$ that are bounded on the bottom by $(0,0)$. This is because translating a box by a lattice point does not change its area nor how many points it contains. Thus, we are only interested in boxes of the following type.

\begin{notation}
Let $\mathcal B=\{B_n=(\alpha_n,\beta_n)\times(0,\widetilde\alpha_n+\widetilde\beta_n):n\in\Z\}$ be the set of maximal empty axis-parallel boxes that are bounded by $(0,0)$ on the bottom and ordered according to height. 
\end{notation}

The progression from $B_n\in\mathcal B$ to $B_{n+1}$ (one box higher) has two cases which depend on which side of the $y$-axis is the point bounding $B_n$ from above. If $\alpha_n+\beta_n>0$, then the top point $(\alpha_n+\beta_n,\widetilde\alpha_n+\widetilde\beta_n)$ of $B_n$ is on the right side of the $y$-axis. To move to the next box we push the right side of $B_n$ in until it reaches its top point, then extend the top  up until it hits another point, by Lemma \ref{lem:TopPointIsDetermined} the top of $B_{n+1}$ is completely determined by the left and right points. Thus the right point of $B_{n+1}$ was the top point of $B_n$, i.e.\ $\beta_{n+1}=\alpha_n+\beta_n$ while the left points of $B_n$ and $B_{n+1}$ are the same, i.e.\ $\alpha_{n+1}=\alpha_n$.  Similarly, if $\alpha_n+\beta_n<0$, then the top point of $B_n$ is on the left side of the $y$-axis and $\alpha_{n+1}=\alpha_n+\beta_n$ while $\beta_{n+1}=\beta_n$. We summarize this by
\begin{align}\label{eq:Going Up}
    \alpha_{n+1} &=\begin{cases}
    \alpha_n            &\text{if } \alpha_n+\beta_n >0,\\
    \alpha_n+\beta_n    &\text{if } \alpha_n+\beta_n <0,
    \end{cases}\\\label{eq:FormulaForBetaNPlus1}
   \beta_{n+1} &=\begin{cases}
    \alpha_n+\beta_n     &\text{if } \alpha_n+\beta_n >0,\\
    \beta_n              &\text{if } \alpha_n+\beta_n <0.
    \end{cases}
\end{align}

The progression from $B_n$ to $B_{n-1}$ (one box lower)  again has two cases which depend on which of the points bounding $B_n$ on the left and right is higher. If $\widetilde\alpha_n-\widetilde\beta_n>0$, then the point bounding $B_n$ on the left is higher. In this case we push the top of $B_n$ down until it reaches the point of the left and then extend the box to the left; the point on the right remains the same. Thus $\beta_{n+1}=\beta_n$ and the new top is $\widetilde\alpha_n=\widetilde\alpha_{n+1}+\widetilde\beta_{n+1}=\widetilde\alpha_{n}+\widetilde\beta_{n}$, which means that $\alpha_{n+1}=\alpha_n-\beta_n$. Similarly, if $\widetilde\alpha_n-\widetilde\beta_n<0$, then the point bounding $B_n$ on the right is higher and we have $\alpha_{n+1}=\alpha_n$ and $\beta_{n+1}=\beta_n-\alpha_n$. We summarize this by
\begin{align}\label{eq:Going Down}
     \alpha_{n-1} &=\begin{cases}
    \alpha_n-\beta_n &\text{if } \widetilde\alpha_n-\widetilde\beta_n >0,\\
    \alpha_n         &\text{if } \widetilde\alpha_n-\widetilde\beta_n <0,
    \end{cases} \\
    \beta_{n-1} &=\begin{cases}
    \beta_n          &\text{if }\widetilde\alpha_n-\widetilde\beta_n >0,\\
    \beta_n-\alpha_n &\text{if }\widetilde\alpha_n-\widetilde\beta_n <0.
    \end{cases}
\end{align}

\begin{lemma}\label{lem:EquivalentLattice}
Let $\Lambda$ be an irrational lattice. There exist $\Delta>1$ and $-1<\widetilde\Delta<0$ such that the lattice $\Lambda'$ generated by
\[
\begin{pmatrix}
1   &-\Delta\\
1   &-\widetilde\Delta
\end{pmatrix}
\quad \text{satisfies}\quad
\frac{\disp(\Lambda')}{\det(\Lambda')}=\frac{\disp(\Lambda)}{\det(\Lambda)}.
\]
\end{lemma}

\begin{proof}
Let $\mathcal B$ be the maximal empty axis-parallel boxes amidst $\Lambda$ bounded by $(0,0)$ on the bottom ordered according to height. We claim that there is an $k$ such that $\alpha_k+\beta_k<0<\widetilde\beta_k-\widetilde\alpha_k$. If $\alpha_0+\beta_0>0$ then the recurrence for going up is $\alpha_{n+1}=\alpha_n$ and $\beta_{n+1}=\alpha_n+\beta_n$ until we reach the minimal $k\in\N$ such that $(k+1)\alpha_0+\beta_0<0$. Also, until $k$, the recurrence for going down is $\alpha_{n-1}=\alpha_n$ and $\beta_{n-1}=\beta_n-\alpha_n$ so that $\widetilde \beta_n - \widetilde \alpha_n>0$. Therefore, this choice of $k$ satisfies our claim.
If $\alpha_0+\beta_0<0$ and $\widetilde\alpha_0-\widetilde\beta_0>0$ then the recurrence for going down
is $\alpha_{n-1}=\alpha_n-\beta_n$ and $\beta_{n-1}=\beta_n$ until we reach the minimal $k\in\N$ such that $\widetilde\alpha_0-(k+1)\widetilde\beta_0<0$. Also, again until $k$, the recurrence for going up is $\alpha_{n+1}=\alpha_n+\beta_n$ and $\beta_{n+1}=\beta_n$ so that $\alpha_n+\beta_n<0$. Thus, this choice of $k$ satisfies our claim, once again.
By Lemma \ref{lem:TopPointIsDetermined}, $\Lambda$ is generated by the matrix
\[
\begin{pmatrix}
\beta_k         &   \alpha_n\\
\widetilde\beta_k  &   \widetilde\alpha_n
\end{pmatrix},
\quad \text{while}\quad
\begin{pmatrix}
\beta_k^{-1}         &   0\\
0  &   \widetilde\beta_k^{-1}
\end{pmatrix}
\begin{pmatrix}
\beta_k         &   \alpha_k\\
\widetilde\beta_k  &   \widetilde\alpha_k
\end{pmatrix}
\]
generates a lattice $\Lambda'$ with the same normalized dispersion as $\Lambda$. By choice of $k$ we have $\Delta=-\alpha_k/\beta_k>1$ and $-1<\widetilde\Delta=-\widetilde\alpha_k/\widetilde\beta_k<0$.
\end{proof}

The following lemma allows us to fix a starting box. We prove it for more general $\Delta$ and $\widetilde\Delta$ as it will be needed in Section \ref{sec:Periodic Disperson}. 

\begin{lemma}\label{lem:StartingBoxIsEmpty}
Let $\Delta\ge 1$ and $-1\le\widetilde\Delta<0$ and let $\Lambda$ be the lattice generated by
$\{(1,1),(-\Delta,-\widetilde\Delta)\}$. Then the box $B_0=(-\Delta,1)\times(0,1-\widetilde\Delta)$ is a maximal empty axis-parallel box amidst the points of $\Lambda$ that is bounded on the bottom by $(0,0)$.
\end{lemma}

\begin{proof}
Let $P$ be the fundamental parallelepiped generated by the basis $\{(1,1),(-\Delta,-\widetilde{\Delta})\}$. The points from $\Lambda$ that describe $P$ are on the boundary of $B_0$ and so $B_0\setminus P$ has four components, each containing a different corner of $B_0$ (see Figure \ref{fig:Polygons}). We claim that the components of $B_0\setminus P$ that contain the lower left and upper right corners, which we call $L$ and $R$ respectively, are contained in the interior of the convex region $P+x(1,1)$. Since the interior of  $P+x(1,1)$ is empty, so too will be $L$ and $R$. 

The region $L$ is open and convex with extreme points $(0,0)$, $(-\Delta,-\widetilde{\Delta})$, and $(-\Delta,0)$. The first two are in $P$ and are therefore trivially in $P+x(1,1)$. Decomposing $(-\Delta,0)$ as
\[
(-\Delta,0)=\tfrac{\Delta}{\Delta-\widetilde\Delta}(-\Delta,-\widetilde\Delta)+\tfrac{\Delta\widetilde\Delta}{\Delta-\widetilde\Delta}(1,1)
\]
and applying the assumptions $\Delta 1$ and $\widetilde\Delta 0$ we see that $(-\Delta,0)$ is indeed in $P+x(1,1)$. As the extreme points of $L$ are contained in $P+x(1,1)$, so too must the open set $L$ be in the interior of $P+x(1,1)$. 
Similarly, the region $R$ is open and convex with extreme points $(1,1)$, $(1-\Delta,1-\widetilde\Delta)$, and $(1,1-\widetilde\Delta)=(0,-\widetilde\Delta)+(1,1)$. The first two points are again trivially in $P+x(1,1)$ while $(0,\widetilde\Delta)$ may be decomposed as
\[
(0,-\widetilde\Delta)=-\tfrac{\widetilde\Delta}{\Delta-\widetilde\Delta}(-\Delta,-\widetilde\Delta)-\tfrac{\Delta\widetilde\Delta}{\Delta-\widetilde\Delta}(1,1).
\]
It follows that $R$ is contained within the interior of $P+x(1,1)$.

Using a similar argument one can show that the components of $B_0\setminus P$ that contain the upper left and lower right corners of $B$ are contained in the interior of $P+x(-\Delta,-\widetilde\Delta)$ and are, therefore, also empty.
\end{proof}

We summarize the previous two Lemmas in the following proposition.

\begin{prop}\label{prop:NBAequivmatrices}
Every irrational matrix is NBA equivalent to a matrix of the form
\[
\begin{pmatrix}
1   &-\Delta\\
1   &-\widetilde\Delta
\end{pmatrix},
\]
where $\Delta>1$ and $-1<\Delta<0$. Amidst the points of a lattice generated by such a matrix, the box $B_0=(-\Delta,1)\times(0,1-\widetilde\Delta)$ is a maximal empty axis-parallel box.
\end{prop}

The upcoming proposition applies to all lattices but as we see in Remark \ref{rem:NormInteger} it is most interesting when applied to quadratic lattices as the appearing summands will be integers.

\begin{prop}\label{prop:VolumeIsNormsPlusConstant}
Let $\Lambda$ be the irrational lattice generated by $(1,1)$ and $(-\Delta,-\widetilde\Delta)$ with $1<\Delta$ and $-1<\tilde\Delta<0$ and let $\mathcal B$ be as above. Then
\[
\vol(B_n)= |\alpha_{n}\widetilde\alpha_{n}|+|\beta_{n}\widetilde\beta_{n}|+\Delta-\widetilde\Delta.
\]
\end{prop}

\begin{proof}
We refer the reader to Figure \ref{fig:Rearange Boxes} to help visualize this proof. We will prove the statement for $n\geq 0$ by induction, the proof for $n\leq0$ is similar. Given $n\geq0$ we partition $B_n$, up to a set of measure zero, into the following four sub-boxes:
\begin{itemize}
    \item $B(\alpha_n)=(\alpha_n,0)\times(0,\widetilde\alpha_n)$, the box whose lower right corner is $(0,0)$ and whose upper left corner is $(\alpha_n,\widetilde\alpha_n)$. This box has volume $\alpha_n\widetilde\alpha_n$.
    \item $B(\beta_n)=(0,\beta)\times(0,\widetilde\beta_n)$, the box whose lower left corner is $(0,0)$ and whose upper right corner is $(\beta_n,\widetilde\beta_n)$. This box has volume $\beta_n\widetilde\beta_n$.
    \item$C_{n,1}=(\alpha_n,0)\times(\widetilde\alpha_n,\widetilde\alpha_n+\widetilde\beta_n)$, the part of $B_n$ that is above $B(\alpha_n)$ and to the left of the $y$-axis.
    \item $C_{n,2}=(0,\beta_n)\times(\widetilde\beta_n,\widetilde\alpha_n+\widetilde\beta_n)$, the part of $B_n$ that is above $B(\beta_n)$ and to the right of the $y$-axis.
\end{itemize}
The box $B_0=(-\Delta,1)\times(0,1-\widetilde\Delta)$ has volume
\[
(1+\Delta)(1-\widetilde\Delta)=|\Delta\widetilde\Delta|+1+\Delta-\widetilde\Delta=|\alpha_{0}\widetilde\alpha_{0}|+|\beta_{0}\widetilde\beta_{0}|+\vol(C_{0,1})+\vol(C_{0,2})
\]
and it follows that $\vol(C_{0,1})+\vol(C_{0,2})=\Delta-\widetilde\Delta$.

Assume that $\vol(C_{n,1})+\vol(C_{n,2})=\Delta-\widetilde\Delta$ for some $n\geq0$. Without loss of generality assume that $\alpha_n+\beta_n>0$, i.e.\ $\widetilde\alpha_{n+1}-\widetilde\beta_{n+1}<0$ and $(\alpha_n,\widetilde\alpha_n)$ is below the point $(\beta_n,\widetilde\beta_n)$. Up to a set of measure zero, we partition $B_{n+1}$ into the following four sub-boxes: $B(\alpha_{n+1})$ and $B(\beta_{n+1})$ together with
\begin{itemize}
\item $D_{n+1,1}=(\alpha_{n+1},0)\times(\widetilde\alpha_{n+1},\widetilde\beta_{n+1})$, the part of $B_{n+1}$ that is above $(\alpha_{n+1},\widetilde\alpha_{n+1})$, below $(\beta_{n+1},\widetilde\beta_{n+1})$, and to the left of the $y$-axis and
\item $D_{n+1,2}=(\alpha_{n+1},\beta_{n+1})\times(\widetilde\beta_{n+1},\widetilde\alpha_{n+1}+\widetilde\beta_{n+1})$, the part of $B_{n+1}$ that is above $(\beta_{n+1},\widetilde\beta_{n+1})$.
\end{itemize}
Using the formulas for $\alpha_{n+1}$ and $\beta_{n+1}$ we see that $D_{n,1}=C_{n,1}$ and 
\[
D_{n+1,2}=(\alpha_n,\alpha_n+\beta_n)\times(\widetilde\alpha_n+\widetilde\beta_n,2\widetilde\alpha_n+\widetilde\beta_n)=C_{n,2}+(\alpha_n,\widetilde\alpha_n).
\]
Thus 
\[
\vol(B_{n+1})=|\alpha_{n+1}\widetilde\alpha_{n+1}|+|\beta_{n+1}\widetilde\beta_{n+1}|+\vol(C_{n,1})+\vol(C_{n,2})=|\alpha_{n+1}\widetilde\alpha_{n+1}|+|\beta_{n+1}\widetilde\beta_{n+1}|+\Delta-\widetilde\Delta
\]
and the induction hypothesis must hold for $n+1$.
\end{proof}

%%%%%%%%%%%%%%%%%%%%%%%%%%%%%%%%%%%%%%%%%%%%%%%%%%%%%%%%%%%%%%%%%%%%%

\section{The continued fraction connection}\label{sec:contfracconnect}

Throughout this section and the next, $\Lambda$ will denote an irrational lattice with generating matrix
\[
\begin{pmatrix}
1   &-\Delta\\
1   &-\widetilde\Delta
\end{pmatrix},
\]
where $1<\Delta$ and $-1<\widetilde\Delta<0$; by Lemma \ref{lem:NBAequivoperations} and Proposition \ref{prop:NBAequivmatrices} this can be done without losing generality.   The set $\mathcal B=\{B_n=(\alpha_n,\beta_n)\times(0,\widetilde\alpha_n+\widetilde\beta_n):n\in\Z\}$ of maximal empty axis-parallel boxes that are bounded by $(0,0)$ on the bottom and ordered according to height will be indexed so that $B_0=(-\Delta,1)\times(0,1-\widetilde\Delta)$; by Lemma \ref{lem:StartingBoxIsEmpty}, the conditions on the generators of $\Lambda$ guarantee that the box $B_0$ is in $\mathcal B$.

We write the continued fraction expansions of $\Delta$ and $-\widetilde\Delta$ as $\Delta=[a_0,a_1,a_2\dots]$ and $-\widetilde\Delta=[0,a_{-1},a_{-2},\dots]$, where $0<a_i\in\N$; the conditions $1<\Delta$ and $-1<\widetilde\Delta<0$ tell us that their continued fraction expansions have this form. The sequence $(a_i)_{i\in\Z}$ induce sequences $(p_i)_{i\in\Z}$ and $(q_i)_{i\in\Z}$ defined by
\begin{align*}
p_{i-1}&=-a_{i}p_{i}+p_{i+1},   &    p_{-1}&=0,& p_{0}&=1, &p_{i+1}&=a_{i}p_{i}+p_{i-1},\\
q_{i-1}&=-a_{i}q_{i}+q_{i+1},   &    q_{-1}&=1, &q_{0}&=0, &q_{i+1}&=a_{i}q_{i}+q_{i-1}.
\end{align*}

\begin{remark}
With these definitions we can now write the initially used matrix as
\[
\begin{pmatrix}
1   &-\Delta\\
1   &-\widetilde\Delta
\end{pmatrix}
=
\begin{pmatrix}
p_0-q_0\Delta   &p_{-1}-q_{-1}\Delta\\
p_0-q_0\widetilde\Delta   &p_{-1}-q_{-1}\widetilde\Delta
\end{pmatrix}.
\]
\end{remark}
The quotient $p_i/q_i$ is the \emph{$i$-th convergent of $\Delta$}. If the $i$-th convergent of $-\widetilde\Delta$ is denoted by $\widetilde p_i/\widetilde q_i$, then $\widetilde p_i=(-1)^{i}p_{-i}$ and $\widetilde q_i=(-1)^{i+1} q_{-i}$. This is easy to see by induction because for negative $i$, 
\begin{align*}
\widetilde p_{i+1} &= \widetilde a_i \widetilde p_i + \widetilde p_{i-1}
&
\widetilde q_{i+1} &= \widetilde a_i \widetilde q_i + \widetilde q_{i-1}\\
                &= (-1)^{i}a_{-i} p_{-i} + (-1)^{i-1} p_{-i+1}
                &
                &= (-1)^{i+1}a_{-i} p_{-i} + (-1)^i q_{-i+1}\\
                &= (-1)^{i+1}(-a_{-i} p_{-i} +  p_{-i+1})
                &
                                &= (-1)^{i+2}(-a_{-i} q_{-i} +  q_{-i+1})\\
                &= (-1)^{i+1}p_{-i-1}
                &
                               &= (-1)^{i+2}q_{-i-1}.
\end{align*}
The following properties about the convergents of $\Delta$ are well known
\begin{gather}
    \frac{p_i}{q_i}         =[a_0,\dots,a_{i-1}],\label{eq:Convergents Fact 1}\\
     p_iq_{i-1} - p_{i-1}q_i  =(-1)^i,\label{eq:Convergents Fact 2}\\
    \frac{p_{2i+1}}{q_{2i+1}}-\Delta<0<\frac{p_{2i}}{q_{2i}}-\Delta\quad \text{or}\quad p_{2i+1}-q_{2i+1}\Delta<0<p_{2i}-q_{2i}\Delta,\text{ and}\label{eq:Convergents Fact 3}\\
    \frac{1}{(a_{i}+2)q_i^2}<\Big|\frac{p_i}{q_i}-\Delta\Big|<\frac{1}{a_{i}q_i^2}\quad \text{or}\quad
    \frac{1}{(a_{i}+2)|q_i|}<|p_{i}-q_{i}\Delta|<\frac{1}{a_{i}|q_i|}.\label{eq:Convergents Fact 4}
\end{gather}
    
Since $\widetilde p_i/\widetilde q_i=-p_{-i}/q_{-i}$ we may replace $-\Delta$ with $-\widetilde\Delta$ in the last two facts when $i<0$.

\begin{remark}
When looking at the next proposition it is helpful to know that a \emph{semiconvergent} of $\Delta=[a_0,a_1,a_2,\dots]$ is a fraction of the form
\[
\frac{jp_i+p_{i-1}}{jq_i+q_{i-1}}
\]
for some $0<j<a_i$. These $a_i-1$ number form a strictly monotone sequence.
\end{remark}

The following proposition relates the sequences $(a_i)_{i\in\Z}$, $(p_i)_{i\in\Z}$, and $(q_i)_{i\in\Z}$ to the boxes in $\mathcal B$. The connection between continued fractions and the maximal empty axis-parallel boxes amidst $\Lambda$ allows us to use the rich theory of the former to study the latter. The main idea is that if $(p-q\Delta,p-q\widetilde\Delta)$ bounds $B_n\in\mathcal B$ on either the left or the right and $q>0$, then there can be no point with a smaller value of $q$ whose first coordinate has both the same sign and is smaller than $p-q\Delta$. This means that, for $n>0$, the points bounding $B_n$ should correspond to either convergents or semiconvergents of $\Delta$.  The proof simply demonstrates that the upward and downward box progression rules in \eqref{eq:Going Up} and \eqref{eq:Going Down} are just the continued fraction algorithm in disguise.

\begin{notation}
Given a sequence $(a_i)_{i\in\Z}$ we define $A_i=\sum_{k=0}^{i-1}a_{k}$ when $i\geq0$ and $A_i=-\sum_{k=i}^{-1}a_k$ when $i<0$.
\end{notation}
\begin{prop}\label{prop:BoxesAreSemiConvergents}
Let $\Lambda$ be a lattice as defined above with the accompanying framework.  Writing $n\in\Z$ as $n=A_i+j$, where $0\leq j<a_i$, then the left and right sides of the $n$-th box in $\mathcal B$ are given by
\begin{align*}
    \alpha_n&=\begin{cases}
     p_{i}-q_{i}\Delta  &\text{if } i \text{ is odd}, \\
    j(p_{i}-q_{i}\Delta) + (p_{i-1}-q_{i-1}\Delta)  &\text{if } i \text{ is even},
    \end{cases}\\
    \beta_n&=\begin{cases}
    j(p_{i}-q_{i}\Delta) + (p_{i-1}-q_{i-1}\Delta)  &\text{if } i \text{ is odd},\\
    p_{i}-q_{i}\Delta  &\text{if } i \text{ is even}.
    \end{cases}
\end{align*}

\end{prop}

\begin{proof}
Let $n=A_i+j$ where $0\leq j < a_i$ as in the statement of the proposition.
We prove the statement first for $n\geq0$. Let $\alpha_n'$ and $\beta_n'$ be the claimed values for $\alpha_n$ and $\beta_n$ respectively. Then $\alpha_{n+1}'$ and $\beta_{n+1}'$ satisfy
\begin{align*}
    \alpha_{n+1}' &=\begin{cases}
    \alpha_n'            &\text{if } i \text{ is odd},\\
    \alpha_n'+\beta_n'    &\text{if } i \text{ is even},
    \end{cases}\\
   \beta_{n+1}' &=\begin{cases}
    \alpha_n'+\beta_n'     &\text{if } i \text{ is odd},\\
    \beta_n'              &\text{if } i \text{ is even},
    \end{cases}
\end{align*}
when $0\leq j<a_i$. The statement is true for $n=0$ by definition. To prove the result by induction, it is sufficient to show that if $0<j<a_i$, then $\alpha_n'+\beta_n'>0$ when $i$ is odd and $\alpha_n'+\beta_n'<0$ when $i$ is even. 
Whenever $i$ is odd, we see by \eqref{eq:Convergents Fact 3} that $p_{i}-q_{i}\Delta<0$ and 
\begin{align*}
    \alpha_n'+\beta_n'&=(j+1)(p_{i}-q_{i}\Delta)+p_{i-1}-q_{i-1}\Delta\\
    &\ge a_i(p_{i}-q_{i}\Delta)+p_{i-1}-q_{i-1}\Delta\\
    &=p_{i+1}-q_{i+1}\Delta>0,
\end{align*}
on the other hand, when $i$ is even, $p_{i}-q_{i}\Delta>0$ and
\begin{align*}
\alpha_n'+\beta_n'&=(j+1)(p_{i}-q_{i}\Delta)+p_{i-1}-q_{i-1}\Delta\\
    &\le a_i(p_{i}-q_{i}\Delta)+p_{i-1}-q_{i-1}\Delta\\
    &=p_{i+1}-q_{i+1}\Delta<0.
\end{align*}

Now assume $n<0$. Let $\alpha_n'$ and $\beta_n'$ be again the claimed values for $\alpha_n$ and $\beta_n$ respectively. Then $\alpha_{n-1}'$ and $\beta_{n-1}'$ satisfy
\begin{align*}
    \alpha_{n-1}' &=\begin{cases}
    \alpha_n'            &\text{if } i \text{ is odd},\\
    \alpha_n' - \beta_{n-1}'  = \alpha_n' - \beta_n'   &\text{if } i \text{ is even},
    \end{cases}\\
   \beta_{n-1}' &=\begin{cases}
    \beta_n' - \alpha_{n-1}'= \beta_n' - \alpha_n'     &\text{if } i \text{ is odd},\\
    \beta_n'            &\text{if } i \text{ is even},
    \end{cases}
\end{align*}
when $0<j\leq a_{i-1}$. The inequalities bounding $j$ are switched since we are going down and thus starting at $A_{i-1}$, the statement follows from going down to $A_i$ and then going back up. The statement is true for $n=0$ by definition. To prove the result by induction, it is sufficient to show that if $0<j\leq a_{i-1}$, then $\widetilde\alpha_n'-\widetilde\beta_n'>0$ when $i$ is even and $\widetilde\beta_n'-\widetilde\alpha_n'>0$ when $i$ is odd. Note that we always have $(p_{i}-q_{i}\widetilde\Delta)>0$.
Whenever $i$ is even, we get 
\begin{align*}
    \widetilde \alpha_n'-\widetilde \beta_n'&=(j-1)(p_{i}-q_{i}\widetilde\Delta)+p_{i-1}-q_{i-1}\widetilde\Delta>p_{i-1}-q_{i-1}\widetilde\Delta>0,
\end{align*}
on the other hand, when $i$ is odd, we get
\begin{align*}
    \widetilde \beta_n'-\widetilde         \alpha_n'&=(j-1)(p_{i}-q_{i}\widetilde\Delta)+p_{i-1}-q_{i-1}\widetilde\Delta>(p_{i-1}-q_{i-1}\widetilde\Delta)>0 \qedhere
\end{align*} 
\end{proof}

Given a two-sided sequence of natural numbers $\mathfrak{A}=(a_i)_{i\in\Z}$ we obtain an irrational lattice by setting $\Delta=[a_0,a_1,a_2\dots]$ and $-\widetilde\Delta=[0,a_{-1},a_{-2},\dots]$. The next couple lemmas tell what happens when we re-index the sequence.

\begin{notation}
Given $\Delta=[a_0,a_1,a_2\dots]$ and $-\widetilde\Delta=[0,a_{-1},a_{-2},\dots]$ with $a_i\in\N$ we define $\Delta_i$ and $-\widetilde\Delta_i$ as $\Delta_i=[a_i,a_{i+1},\dots]$ and $-\widetilde\Delta_i=[0,a_{i-1},a_{i-2},\dots]$.
\end{notation}

\begin{lemma}\label{lem: Delta i}
Let $\Delta=[a_0,a_1,a_2,\dots]$ and $-\widetilde\Delta = [0,a_{-1},a_{-2}\dots]$ with $a_i\geq1$.  Then
\[
-\Delta_i=-\frac{-p_{i-1}+q_{i-1}\Delta}{p_{i}-q_{i}\Delta}=\frac{p_{i-1}-q_{i-1}\Delta}{p_{i}-q_{i}\Delta}
\quad\text{and}\quad
-\widetilde\Delta_i=\frac{p_{i-1}-q_{i-1}\widetilde\Delta}{ p_{i}-q_{i}\widetilde\Delta}.
\]
\end{lemma}

\begin{proof}
We are proving the assertion for $-\widetilde\Delta_i$, since the case of $-\Delta_i$ is widely known and works out similar.
The claimed equality is obviously true for $i=0$. Assuming this equality is true for a given $i$ we show that they are then also true for $i+1$ and $i-1$.
\begin{align*}
-\widetilde\Delta_{i+1}&=[0,a_{i+1}, a_{i},a_{i-1},\dots]
&
 -\widetilde\Delta_{i-1}&=[0,a_{i-2}, a_{i-3},\dots]\\
                &=\frac{1}{[a_{i+1},a_{i}, a_{i-1},\dots]}
                &
                &=[a_{i-1},a_{i-2}, a_{i-3},\dots]-a_{i-1}\\
                &=\frac{1}{[0,a_{i}, a_{i-1},\dots]+a_{i+1}}
                &
                &=\frac{1}{[0, a_{i-1},a_{i-2},a_{i-3},\dots]}-a_{i-1}\\
                &=\frac{1}{-\widetilde\Delta_i+a_{i+1}}
                &
                &=\frac{1}{-\widetilde\Delta_{i}}-a_{i-1}\\
                &=\frac{1}{\frac{p_{i-1}-q_{i-1}\widetilde\Delta}{ p_{i}-q_{i}\widetilde\Delta}+\frac{a_i(p_{i}-q_{i}\widetilde\Delta)}{ p_{i}-q_{i}\widetilde\Delta}}
                &
                &=\frac{p_{i}-q_{i}\widetilde\Delta}{p_{i-1}-q_{i-1}\widetilde\Delta}-\frac{a_i( p_{i-1}-q_{i-1}\widetilde\Delta)}{p_{i-1}-q_{i-1}\widetilde\Delta}\\
                &=\frac{1}{\frac{p_{i+1}-q_{i+1}\widetilde\Delta}{ p_{i}-q_{i}\widetilde\Delta}}
                &
                &=\frac{p_{i-2}-q_{i-2}\widetilde\Delta}{p_{i-1}-q_{i-1}\widetilde\Delta}\\
                &=\frac{ p_{i}-q_{i}\widetilde\Delta}{p_{i+1}-q_{i+1}\widetilde\Delta}
\end{align*}
In both cases we used the induction hypothesis on the $4$th line.
\end{proof}

\begin{notation}
Let $\Lambda$ be an irrational lattice as defined above with the accompanying framework. For each $i\in\Z$, we define $\Gamma_i$ to be the lattice with generating matrix
\[
\begin{pmatrix}
1   &-\Delta_i\\
1   &-\widetilde\Delta_i
\end{pmatrix}
\]
and $\mathcal C=\{C_n:n\in\Z\}$ to be the set of maximal empty axis-parallel boxes amidst $\Gamma_i$ that are bounded by $(0,0)$ on the bottom and ordered according to height so that $C_0=(-\Delta_i,1)\times(0,1-\widetilde\Delta_i)$.
\end{notation}

\begin{lemma}\label{lem:Can start anywhere}
Let $\Lambda$ be an irrational lattice as defined above with the accompanying framework with $\Gamma_i$ and $\mathcal C$ defined as above.  Then for all $j\in\Z$,
\[
\frac{\vol(C_{n-A_i})}{\det(\Gamma_i)}=\frac{\vol(B_{n})}{\det(\Lambda)}.
\]
In particular $\Lambda$ and $\Gamma_i$ have the same normalized dispersion.
\end{lemma}

\begin{proof}
Observe that $(1,1)=(p_0-q_0\Delta,p_0-q_0\widetilde\Delta)$ and $(-\Delta,-\widetilde\Delta)=(p_{-1}-q_{-1}\Delta,p_{-1}-q_{-1}\widetilde\Delta)$. Let $n=A_k+j$ with $0\leq j<a_k$. One can easily verify 
\[
\begin{pmatrix}
  j(p_{k}-q_{k}\Delta) + (p_{k-1}-q_{k-1}\Delta) & p_{k}-q_{k}\Delta \\
  j(p_{k}-q_{k}\widetilde\Delta) + (p_{k-1}-q_{i-1}\widetilde\Delta) & p_{k}-q_{k}\widetilde\Delta
\end{pmatrix}=\begin{cases}
\begin{pmatrix}
  1  &   -\Delta\\
 1  &   -\widetilde\Delta
\end{pmatrix}
\Bigg[
\prod_{r=0}^{k-1}
\begin{pmatrix}
  a_r  &   1\\
 1  &   0
\end{pmatrix}
\Bigg]
\begin{pmatrix}
  j  &   1\\
 1  &   0
\end{pmatrix}
&\text{if } k\geq0, \\[12pt]
\begin{pmatrix}
  1  &   -\Delta\\
 1  &   -\widetilde\Delta
\end{pmatrix}
\Bigg[
\prod_{r=k}^{-1}
\begin{pmatrix}
  a_r  &   1\\
 1  &   0
\end{pmatrix}
\Bigg]^{-1}
\begin{pmatrix}
  j  &   1\\
 1  &   0
\end{pmatrix}
&\text{if } k<0.
\end{cases}
\]
If $B=(b_{s,t})$ is the above matrix, then
$
\vol(B_{n})=|b_{1,1}-b_{1,2}|(b_{2,1}+b_{2,2}).
$
Similarly the matrix
\[
C:=
% \begin{pmatrix}
%  p_{k}-q_{k}\Delta & j(p_{k}-q_{k}\Delta) + (p_{k-1}-q_{k-1}\Delta)\\
%  p_{k}-q_{k}\widetilde\Delta & j(p_{k}-q_{k}\widetilde\Delta) + (p_{k-1}-q_{i-1}\widetilde\Delta)
%\end{pmatrix}
\begin{cases}
\begin{pmatrix}
  1  &   -\Delta_i\\
 1  &   -\widetilde\Delta_i
\end{pmatrix}
\Bigg[
\prod_{r=i}^{k-1}
\begin{pmatrix}
  a_r  &   1\\
 1  &   0
\end{pmatrix}
\Bigg]
\begin{pmatrix}
  j  &   1\\
 1  &   0
\end{pmatrix}
&\text{if } i\leq k, \\[12pt]
\begin{pmatrix}
  1  &   -\Delta_i\\
 1  &   -\widetilde\Delta_i
\end{pmatrix}
\Bigg[
\prod_{r=k}^{i-1}
\begin{pmatrix}
  a_r  &   1\\
 1  &   0
\end{pmatrix}
\Bigg]^{-1}
\begin{pmatrix}
  j  &   1\\
 1  &   0
\end{pmatrix}
&\text{if } k<i
\end{cases}
\]
satisfies 
$
\vol(C_{n-A_i})=|c_{1,1}-c_{1,2}|(c_{2,1}+c_{2,2}).
$
But
\[
\begin{pmatrix}
  1  &   -\Delta_i\\
 1  &   -\widetilde\Delta_i
\end{pmatrix}=\begin{cases}
\begin{pmatrix}
 p_{i}-q_{i}\Delta  &0\\
 0          & p_{i}-q_{i}\widetilde\Delta
\end{pmatrix}^{-1}
\begin{pmatrix}
  1  &   -\Delta\\
 1  &   -\widetilde\Delta
\end{pmatrix}
\prod_{r=0}^{i-1}
\begin{pmatrix}
  a_r  &   1\\
 1  &   0
\end{pmatrix}
&\text{if } i\geq0, \\[12pt]
\begin{pmatrix}
 p_{i}-q_{i}\Delta  &0\\
 0          & p_{i}-q_{i}\widetilde\Delta
\end{pmatrix}^{-1}
\begin{pmatrix}
  1  &   -\Delta\\
 1  &   -\widetilde\Delta
\end{pmatrix}
\Bigg[
\prod_{r=i}^{-1}
\begin{pmatrix}
  a_r  &   1\\
 1  &   0
\end{pmatrix}
\Bigg]^{-1}
&\text{if } i<0
\end{cases}
\]
so that $C=\diag(p_{i}-q_{i}\Delta,p_{i}-q_{i}\widetilde\Delta)^{-1}B$ from which the lemma easily follows.
\end{proof}

% \begin{proof}
% By sequentially adding integer multiples of one column of the generating matrix for $\Lambda$ to another we see that $\Lambda$ is generated, up to a column swap, by the matrix
% \[
% \begin{pmatrix}
% p_i-q_i\Delta   &   p_{i-1}-q_{i-1}\Delta \\
% p_i+q_i\widetilde\Delta  &   p_{i-1}+q_{i-1}\widetilde\Delta
% \end{pmatrix}
% \quad\text{which then becomes}\quad
% \begin{pmatrix}
%   1  &   -\Delta_i\\
%  1  &   \widetilde\Delta_i
% \end{pmatrix}
% \]
% when multiplied on the right by $\diag(p_i-q_i\Delta,p_i+q_i\widetilde\Delta)^{-1}$. The columns of the first matrix represent the right and left bounding points of $B_{A_i}$ while the columns of the second matrix represents the right and left bounding points of $C_0$. It follows that  
% \[
% \vol(C_0)\frac{\vol(B_{A_i})}{(p_i-q_i\Delta)(p_i+q_i\widetilde\Delta)}\quad\text{and}\quad
% \det(\Gamma)=\frac{\det(\Lambda)}{(p_i-q_i\Delta)(p_i+q_i\widetilde\Delta)}.
% \]
% \textbf{Complete the proof}
% \end{proof}

Because of the above lemma, for the purposes of normalized dispersion it makes sense to think of a lattice as a sequence $(a_i)_{i\in\Z}$ which can be re-indexed as needed. The next lemma shows that we can also reverse it, We leave the proof as a small exercise to the reader.

\begin{lemma}
Let $(a_i)_{i\in\Z}$ be the sequence associated to the irrational lattice $\Lambda$. Then the sequence $(a_{-i})_{i\in\Z}$ is associated to the lattice with generating matrix
\[
\begin{pmatrix}
1   &   \widetilde\Delta-a_0\\
1   &   \Delta-a_0
\end{pmatrix}.
\]
\end{lemma}

\begin{lemma}\label{lem:Normalized Volume of a Box} Let $\Lambda$ be an irrational lattice with the accompanying framework and the notation as above, $i$ an arbitrary integer, and $0\leq j<a_i$. Then we have
\[
\frac{\vol(B_{A_i+j})}{\det(\Lambda)}=\frac{(1-j+\Delta_i)(1+j-\widetilde\Delta_i)}{\Delta_i-\widetilde\Delta_i}.
\]
Moreover this term, when considered as a function in $\Delta_i$, resp. $-\widetilde\Delta_i$, is increasing for $0<j<a_i$.
\end{lemma}

\begin{proof}
Let $\Gamma_i$ and $\mathcal C$ be as above. By Proposition \ref{prop:BoxesAreSemiConvergents}, $C_j=(j-\Delta_i,1)\times(0,1+j-\widetilde\Delta)$ for $0\leq j\leq a_0$. We apply  Lemma \ref{lem:Can start anywhere} to obtain
\[
\frac{\vol(B_{A_i+j})}{\det(\Lambda)}=\frac{\vol(C_j)}{\det(\Gamma_i)}=\frac{(1-j+\Delta_i)(1+j-\widetilde\Delta_i)}{\Delta_i-\widetilde\Delta_i}.
\]
The derivative of this with respect to $\Delta_i$ is
\begin{align*}
\frac{\partial}{\partial\Delta_i}\frac{(1-j+\Delta_i)(1+j-\widetilde\Delta_i)}{\Delta_i-\widetilde\Delta_i}
%&=\frac{(\Delta_i+\widetilde\Delta_i)(1+j+\widetilde\Delta_i)-(1-j+\Delta_i)(1+j+\widetilde\Delta_i)}{(\Delta_i+\widetilde\Delta_i)^2}\\
&= \frac{(j-1-\widetilde\Delta_i)(1+j-\widetilde\Delta_i)}{(\Delta_i-\widetilde\Delta_i)^2},
\end{align*}
which, for $\Delta_i>1$ and $-1<\widetilde\Delta_i<0$, is positive because $0< j$. Similarly, the derivative with respect to $\widetilde\Delta_i$ is easily seen to be positive since  $ j<a_i$.
\end{proof}
% \begin{proof}
% The lattice generated by the first matrix is also is generated by both
% \[
% \begin{pmatrix}
% p_{2k}-q_{2k}\Delta            &   p_{2k-1}-q_{2k-1}\Delta\\
% p_{2k}+q_{2k}\widetilde\Delta  &   p_{2k-1}+q_{2k-1}\widetilde\Delta
% \end{pmatrix}
% \text{ and }
% \begin{pmatrix}
% p_{2k}-q_{2k}\Delta            &   p_{2k+1}-q_{2k+1}\Delta\\
% p_{2k}+q_{2k}\widetilde\Delta  &   p_{2k+1}+q_{2k+1}\widetilde\Delta
% \end{pmatrix}
% \]
% for each $k\in \Z$, since these matrices are obtained from the canonical generating matrix by sequentially adding integer multiples of one column to another. The lattices generated by 
% \[
% \begin{pmatrix}
% 1            &   \frac{p_{2k-1}-q_{2k-1}\Delta}{p_{2k}-q_{2k}\Delta }\\
% 1  &  \frac{p_{2k-1}+q_{2k-1}\widetilde\Delta}{p_{2k}+q_{2k}\widetilde\Delta }
% \end{pmatrix}
% \text{ and }
% \begin{pmatrix}
% 1& \frac{p_{2k}-q_{2k}\Delta}{ p_{2k+1}-q_{2k+1}\Delta}          \\
% 1&\frac{p_{2k}+q_{2k}\widetilde\Delta}{ p_{2k+1}+q_{2k+1}\widetilde\Delta}        
% \end{pmatrix}
% \]
% have the same normalized dispersion as the lattice $\Lambda$ because they are obtained from $\Lambda$ by multiplying by a diagonal matrix and, in the second case, swapping the columns of the generating matrix. We apply Lemma \ref{lem: Delta i} with $i=2k$ or $i=2k+1$ to complete the proof.
% \end{proof}

The above lemma motivates the following definition.

\begin{defi}
Let $\mathfrak A=(a_i)_{i\in\Z}$ be a sequence of natural numbers. We define the dispersion of the sequence to be
\[
\disp(\mathfrak A)=\sup_{i\in\Z}\max_{0\leq j<a_i}\frac{(1-j+\Delta_i)(1+j-\widetilde\Delta_i)}{\Delta_i-\widetilde\Delta_i},
\]
where $\Delta_i=[a_i,a_{i+1},a_{i+2},\dots]$ and $-\widetilde\Delta_i=[0,a_{i-1},a_{i-2},\dots]$.
\end{defi}

The following proposition follows directly from the previous lemma. It shows that for the purposes of normalized dispersion, we may think of a lattice $\Lambda$ as a two-sided sequence that may be re-indexed as necessary.

\begin{prop}\label{prop:lattices-sequences}
Let $\mathfrak A=(a_i)_{i\in\Z}$ be a sequence of natural numbers. Let $\Lambda$ be the lattice with generators $\Delta=[a_i,a_{i+1},a_{i+2},\dots]$ and $-\widetilde\Delta_i=[0,a_{i-1},a_{i-2},\dots]$ and let $\Gamma_i$ be the lattice with generators $\Delta_i$ and $\widetilde \Delta_i$. Then
\[
\frac{\disp(\Lambda)}{\det(\Lambda)}=\frac{\disp(\Gamma_i)}{\det(\Gamma_i)}=\disp(\mathfrak A).
\]
\end{prop}

Let us summarize the results before into a single corollary.

\begin{cor}\label{cor:sequence correspondence}
There is a one-to-one correspondence between the classes of irrational lattices and two-sided sequences of natural numbers, up to re-indexing or transposing.
\end{cor}

\begin{thm}\label{thm:Genreal upper/lower bounds}
Let $\Lambda$ be a irrational lattice with generators $\{(1,1),(-\Delta,-\widetilde\Delta)\}$ whose continued fraction expansions are of the form $\Delta=[a_0,a_1,a_2,\dots]$ and $-\widetilde\Delta=[0,a_{-1},a_{-2},\dots]$ with $a_k\in\N$. Then for all $i\in\Z$ with $a_i>1$, 
% \[
% A:=\frac{a_i+4}{4}+\frac{1}{a_i}-\frac{r}{4a_i}\leq\max_{0\leq j\leq a_i}\frac{\vol(B_{A_i+j})}{\det(\Lambda)}< \frac{a_i+4}{4}+\frac{3}{4}+\frac{3}{4(a_i+2)}=A+\frac{1}{2}.
% \]
\[
L(a_i):=\frac{a_i}{4}+1+\frac{1}{a_i}-\frac{r}{4a_i}<\max_{0\leq j< a_i}\frac{\vol(B_{A_i+j})}{\det(\Lambda)}< L(a_i+2)<L(a_i)+\frac{1}{2}.
\]
% \[
% \frac{(a_i+2)^2-r}{4a_i}\leq\max_{0\leq j\leq a_i}\frac{\vol(B_{A_i+j})}{\det(\Lambda)}\leq 
% \frac{(a_i+4)^2-r}{4(a_i+2)}
% \]
In particular, $\Lambda$ has finite dispersion if and only if the continued fraction coefficients are bounded.
\end{thm}

\begin{proof}
% Let $\Gamma$ be the lattice generated by $(1,1)$ and $(-\Delta_i,\widetilde\Delta_i)$. Applying Lemma \ref{lem:Can start anywhere} and adopting its notation we see that
% \[
% \max_{0\leq j\leq a_i}\frac{\vol(B_{A_i+j})}{\det(\Lambda)} = \max_{0\leq j\leq a_i}\frac{\vol(C_{j})}{\det(\Gamma)}.
% \]
By Lemma \ref{lem:Normalized Volume of a Box}, 
\begin{equation}\label{eq:Area of C_j}
\frac{\vol(B_{A_i+j})}{\det(\Lambda)}=\frac{(1-j+a_i+\delta)(1+j-\widetilde\Delta_i)}{a_i+\delta-\widetilde\Delta_i},
\end{equation}
where $\delta=[0,a_{i+1},a_{i+2},\dots]=\Delta_i-a_i$. One can easily verify that the area of the $A_i$-th box is less than that of the $(A_i+1)$-th box when $a_i>1$. By Lemma  \ref{lem:Normalized Volume of a Box}, this is increasing in both $\delta$ and $-\widetilde\Delta_i$ for fixed $0<j<a_i$. Since the upper and lower bound for $\delta$ and $-\widetilde\Delta_i$ are $1$ and $0$ (which will never\footnote{The only modification to the proof needed for Section \ref{sec:Periodic Disperson} is to observe that the assumption that the lattice is not $\Z^2$ means that we cannot attain both bounds at the same time.} be obtained), for $0<j<a_i$,
\[
\frac{(1+a_i-j)(1+j)}{a_i}< \frac{\vol(B_{A_i+j})}{\det(\Lambda)} < \frac{(2+a_i-j)(2+j)}{a_i+2}.
\]
Since $j$ can only take integer values the upper and lower bound are maximized at either the floor or ceiling of $a_i/2$. We substitute $j=(a_i\pm r)/2$ where $r\equiv a_i\mod 2$ to get
\begin{align*}
\frac{1}{a_i}\Bigg(1+\frac{a_i}{2}\mp\frac{r}{2}\Bigg)\Bigg(1+\frac{a_i}{2}\pm\frac{r}{2}\Bigg)
    &<
    \max_{0\leq j< a_i}\frac{\vol(B_{A_i+j})}{\det(\Lambda)} 
    < \frac{1}{a_i+2}\Bigg(2+\frac{a_i}{2}\mp\frac{r}{2}\Bigg)\Bigg(2+\frac{a_i}{2}\pm\frac{r}{2}\Bigg)\\
    \frac{1}{a_i}\Bigg(\frac{a_i^2}{4}+a_i+1-\frac{r}{4}\Bigg)
    &<
    \max_{0\leq j< a_i}\frac{\vol(B_{A_i+j})}{\det(\Lambda)}
    <
    \frac{1}{a_i+2}\Bigg(\frac{a_i^2-4}{4}+2(a_i+2)+1-\frac{r}{4}\Bigg)\\
     \frac{a_i}{4}+1+\frac{1}{a_i}-\frac{r}{4a_i}
    &<
    \max_{0\leq j< a_i}\frac{\vol(B_{A_i+j})}{\det(\Lambda)}
    <\frac{a_i+2}{4}+1+\frac{1}{a_i+2}-\frac{r}{4(a_i+2)}.\qedhere
\end{align*}
 \end{proof}

%%%%%%%%%%%%%%%%%%%%%%%%%%%%%%%%%%%%%%%%%%%%%%%%%%%%%%%%%%%%%%%%%%%%%%%%%%%%%%%%%%%%%%%%%%

\section{Dispersion of quadratic (integer) lattices}\label{sec:dispofquadlatt}

A \emph{quadratic field} is a degree 2 field extension of the rationals. They are of the form $\Q[\sqrt{d}]$ where $d\in\Z$ is square-free and consist of \emph{quadratic numbers}, i.e.\ the roots of quadratic polynomials with integer coefficients.  A \emph{quadratic integer} is the root of a monic polynomial of degree 2 with integer coefficients and the ring of all quadratic integers in $\Q[\sqrt{d}]$ is called the \emph{ring of integers}. It takes the from $\Z[\delta_d]$ where 
\[
\delta_d:=\begin{cases}
\sqrt{d}    &\text{if } d\equiv 2\text{ or } 3 \mod 4,\\
\frac{1+\sqrt{d}}{2} &\text{if } d\equiv 1\mod 4.
\end{cases}
\]
Given a quadratic number $\alpha=r+s\sqrt{d}\in\Q[\sqrt{d}]$ we define its conjugate to be $\overline{\alpha}=r-s\sqrt{d}$. The \emph{norm} and \emph{trace} of a quadratic number are defined to be $N(\alpha)=\alpha\overline\alpha$ and $Tr(\alpha)=\alpha+\overline\alpha$. Over $\Q[\sqrt{d}]$ these functions take rational values while over $\Z[\delta_d]$ these functions take integer values. 

\begin{remark}\label{rem:NormInteger}
Since the norm is always an integer, Proposition \ref{prop:VolumeIsNormsPlusConstant} tells us in particular that the fractional part of the area of all empty axix-parallel boxes amidst the points of a quadratic lattice is the same.
\end{remark}

To each subring $R=\Z[\delta]\subseteq\Z[\delta_d]$ of the ring of integers $\Q[\sqrt{d}]$ we associate the lattice 
\[
\Lambda=\{(\alpha,\overline\alpha)\in\R^2:\alpha\in\Z[\delta]\}
\quad\text{generated by}\quad 
\begin{pmatrix}
1   &   \delta\\
1   &   \overline\delta
\end{pmatrix}.
\]
We call such lattices \emph{quadratic lattices}.  As $\delta=m+n\delta_d$, the determinant of a quadratic lattice must be of the form
\begin{equation}\label{eq:determinat}
\det(\Lambda)=\begin{cases}
n\sqrt{d}   &\text{if } d\equiv1\mod4,\\
2n\sqrt{d}  &\text{otherwise}.
\end{cases}
\end{equation}
The square of this determinant is called the \emph{discriminant} of $R$ and is denoted by $\disc(R)$. 

\begin{notation}
We denote the lattice of $\Z[\delta_d]$ by $\Lambda_d$. 
\end{notation} 

It is well known that the continued fraction expansion of a number $\Delta\in\R$ is purely periodic if and only if it is a quadratic number that satisfies 
$1<\Delta$ and $-1<\overline\Delta<0$. We denote the continued fraction expansion of a purely periodic number by $\Delta=[\overline{a_0,a_1,\dots,a_{l-1}}]$, where $l$ is the minimal period length. It is well-known that $-\overline\Delta=[0,\overline{a_1,a_2,\dots,a_l}]$ whenever $\Delta$ is a purely periodic quadratic number.

\begin{lemma}\label{lem:UPPG}
The unique purely periodic generator of the subring $\Z[\delta]\subset\Z[\delta_d]$ is $\Delta=\lfloor-\overline\delta\rfloor+\delta$.
\end{lemma}

Purely periodic numbers are pretty much the only ones that can be handled without a computer. By Theorem \ref{thm:Genreal upper/lower bounds} we only need to check the largest coefficient and any that are smaller only by 1, contained in a single period. We will see in the next theorem that purely quadratic integers are especially easy to handle because the largest coefficient not only occurs exactly once in the period, but is always twice as large as the next largest. Conveniently, the most interesting lattices for dispersion, i.e.\ those found in Theorems \ref{thm:Thight Bound} and \ref{thm:BestLattices}, are generated by purely periodic numbers.

% \begin{proof}
% This is easily adapted from the proof of \cite[6.7.1]{trif}.\jaspar{Check this or actually write this down.}
% \end{proof}

\begin{thm}\label{thm:EstimateOnAi}
If $\Delta=[\overline{a_0,a_1,\dots,a_{l-1}}]$ is  a purely periodic quadratic integer, then for all $0<i<l$,
\[
\qquad a_i \leq \bigg\lfloor\frac{\Delta-\overline\Delta}{2}\bigg\rfloor =   \bigg\lfloor\frac{a_0}{2} - \overline \Delta\bigg\rfloor < \bigg\lfloor\frac{a_0}{2} + \frac{1}{a_{l-1}}  \bigg\rfloor.
\]
In particular,
\[
a_i\leq \Big\lceil \frac{a_0}{2}\Big\rceil.
\]
\end{thm}

\begin{proof}
The proof can be found in Appendix \ref{sec:Proof of EstimateOnAi}.
\end{proof}

Keep in mind for the next theorem that $\disc(R)$ can be computed without continued fractions.

\begin{thm}\label{thm:MainTheorem}
The dispersion of the lattice $\Lambda$ associated to a subring $R$ of the ring of integers of some quadratic field is
\[
\disp(\Lambda)=\Bigg(\frac{\det(\Lambda)}{2}+1\Bigg)^2-\frac{r}{4}=\Bigg(\frac{\sqrt{\disc(R)}}{2}+1\Bigg)^2-\frac{r}{4},
% =\Bigg(\frac{\Delta-\overline\Delta}{2}+1\Bigg)^2-\frac{r}{4},
\]
where $r\equiv\det(\Lambda)^2\mod2$.
\end{thm}

\begin{proof}
Let $\Delta$ be the purely periodic generator of the ring $R$ which may be found using Lemma \ref{lem:UPPG}. Using basic calculus, the maximum of $\vol(B_j)=(1-j+\Delta)(1+j-\overline\Delta)$ over $0\leq j\leq a_0$ occurs at either the floor or ceiling of $(\Delta+\overline\Delta)/2$, which, because $a_0=\Delta+\overline\Delta$ is an integer, is equal to
\begin{equation}\label{eq:max vol(B_j)}
   \max_{0\leq j\leq a_0}\vol(B_j)=\Bigg(1-\frac{\Delta+\overline\Delta}{2}\mp\frac{r}{2}+\Delta\Bigg)\Bigg(1+\frac{\Delta+\overline\Delta}{2}\pm\frac{r}{2}-\overline\Delta\Bigg)=\Bigg(1+\frac{\Delta-\overline\Delta}{2}\Bigg)^2-\frac{r}{2} 
\end{equation}
where $r\equiv a_0\mod2$. We apply Theorem \ref{thm:Genreal upper/lower bounds} and Theorem \ref{thm:EstimateOnAi} to estimate
\begin{align*}
    \vol(B_{A_i+j}) 
    &\leq (\Delta-\overline\Delta)\Bigg(\frac{a_i}{4}+\frac{3}{2}+\frac{1}{a_i+2}-\frac{r}{4(a_i+2)}\Bigg)\\
    &< (\Delta-\overline\Delta)\Bigg(\frac{\Delta-\overline\Delta}{8}+\frac{3}{2}+\frac{2}{\Delta-\overline\Delta+4}-\frac{r}{2(\Delta-\overline\Delta+4)}\Bigg)
\end{align*}
for $i\not\in l\Z$ and find that \eqref{eq:max vol(B_j)} will be greater than this value whenever $(\Delta-\overline\Delta)^2\geq 21$.

By (\ref{eq:determinat}) the discriminant of a subring of $\Z[\delta_d]$, where $d\in\N$ is square free, is of the form $n^2d$ when $d\equiv1\mod 4$ and $4n^2d$ otherwise. The only values of $\Delta$ that make $(\Delta-\overline\Delta)^2<21$ are $\delta_5,\,1+\delta_2,\, 1+\delta_3,\, 1+\delta_{13},\, 1+\delta_{17},\,\text{and } 1+2\delta_5$.  Because $\Delta$ is purely periodic, $\vol(B_{A_l+j})=\vol(B_{j})$ holds for all $0\leq j <A_l$. In Figure \ref{fig:Cases} we plot the values of $\vol(B_j)-(\Delta-\overline\Delta)$, which by Proposition \ref{prop:VolumeIsNormsPlusConstant} is an integer, for $0\leq j\leq A_l$ in each of these cases and we see clearly that the maximum value of $\vol(B_j)-(\Delta-\overline\Delta)$ occurs for some $0\leq j< a_0$, which is what we wanted. 
\end{proof}

Next we will obtain sharp upper and lower bounds on the dispersion of a lattice $\Lambda$ based on the largest continued fraction coefficient of its generators which we denote by $a$. First we observe that since
\[
[\overline{a}]=\frac{a+\sqrt{a^2+4}}{2}
\quad\text{and}\quad
-[0,\overline{a}]=\frac{a-\sqrt{a^2+4}}{2}
\quad\text{both satisfy}\quad
x^2-ax+1=0,\text{ and since}
\]
\[
[\overline{a,1}]=\frac{a+\sqrt{a^2+4a}}{2}\quad\text{and}\quad-[0,\overline{1,a}]=\frac{a-\sqrt{a^2+4a}}{2} \quad \text{both satisfy} \quad x^2-ax+a=0,
\]
 Theorem \ref{thm:MainTheorem} applies to both the lattice corresponding to the sequence $(\overline a)$ and the lattice corresponding to the sequence $(\overline{a,1})$.

\begin{thm}\label{thm:Thight Bound}
Let $\Lambda$ satisfy the framework at the start of Section 3 with associated sequence $\mathfrak{A}=(a_i)_{i\in\Z}$ and suppose that $a = \max \mathfrak{A}$. Then 
\[
\disp(\overline a)\leq\frac{\disp(\Lambda)}{\det(\Lambda)}\leq\disp(\overline{a,1}).
\]
\end{thm}

\begin{proof}
By Lemma \ref{lem:Normalized Volume of a Box}  we see that
\begin{equation}
    \max_{0\leq j<a_i}\frac{\vol(B_{A_i+j})}{\det(\Lambda)}\leq \frac{(1-j+a_i+\delta)(1+j-\overline\Delta_i)}{a_i+\delta-\overline\Delta_i},
\end{equation}
where $\delta=\Delta_i-a_i$, is increasing in $a_i$, $\delta$, and $\overline\Delta_i$ for $0<j<a_i$ and so will be maximized when those values are as large as they can be given the restriction $1\leq a_i\leq a$. This occurs at
\[
a_i=a\quad\text{and}\quad \delta=-\overline\Delta_i=[0,\overline{1,a}],
\]
which after applying Lemma \ref{lem:Normalized Volume of a Box} yields the upper bound.

% Assume that $\Lambda$ is a lattice with $\max a_i=a$. If the sequence $(a_i)_{i\in\Z}$ is eventually all $a$'s in either the positive or negative direction then, by Lemma \ref{lem:Can start anywhere}, we may that either $a_{-1}<a$ and $a_i=a$ for all $i>0$, or $a_1<a$ and $a_i=a$ for all $i<0$. In either case 
% \[
% \frac{\disp(\Lambda_{4+a^2})}{\det(\Lambda_{4+a^2})}<\max_{0\leq j\leq a} \frac{\vol(B_j)}{\det(\Lambda)}.
% \]
% If the sequence is not eventually all $a$'s, then re-index the sequence so that $a_0=a$ and $a_{-1}<a$ and let  $k$ be the smallest positive integer such that $a_k<a$. If $k$ is odd then 
% \[
% \frac{\disp(\Lambda_{4+a^2})}{\det(\Lambda_{4+a^2})}<\max_{0\leq j\leq a} \frac{\vol(B_j)}{\det(\Lambda)}.
% \]
% Suppose that $k\geq2$ is even. 
% \begin{itemize}
% \item Letting $\overline\varphi_a=[0,\overline a]$ then $-\overline\Delta-\overline\varphi_a>\overline\varphi_a-\delta>0$. Somehow, the amount you decrease 
% \[
%     \max_{0\leq j<a_i}\frac{\vol(B_{j})}{\det(\Lambda)}
% \]
% by replacing $\overline\Delta$ with $[0,\overline a]$ is greater than the amount you increase it by replacing $\delta$ with $[0,\overline a]$.\\
% \item Show that we can replace one $a_i<a$ with $a$ without increasing the dispersion. 
% \end{itemize}
% \end{proof}

% \begin{proof}[Alternate proof, lower bound]
Assume that $\Lambda$ is a lattice with associated sequence $\mathfrak{A}$ and $a = \max \mathfrak{A}$. If the sequence $\mathfrak{A}$ is eventually all $a$'s in either the positive or negative direction then, by Corollary \ref{cor:sequence correspondence}, we may assume that $\Delta=[\overline{a}]$ and $a_{-1} < a$ and in particular $\widetilde\Delta>[0,\overline{a}]$, in which case, by Lemma \ref{lem:Normalized Volume of a Box},
\[
\disp(\overline a)<\max_{0\leq j\leq a} \frac{\vol(B_j)}{\det(\Lambda)}.
\]
Assume now that the sequence is not eventually all $a$'s. We are going to show in several steps that $\mathfrak{A} = (a)_{i\in\Z}$ is the optimal choice. From now on we will assume that $a_0=a$ and $a_1\neq a$, which can be done by re-indexing via Corollary \ref{cor:sequence correspondence}. Now assuming $a_1=1$ and all other $a_i$ being arbitrary we see for even $a$
\begin{align*}
    \max_{0\leq j\leq a} \frac{\vol(B_j)}{\det(\Lambda)} 
    &\ge \frac{(1-\frac{a}{2}+\Delta)(1+\frac{a}{2}-\widetilde\Delta)}{\Delta-\widetilde\Delta}.
\end{align*}

In the following we are going to omit very lengthy calculation which can be easily checked with the help of a computer. Knowing that this function is increasing both in $\Delta$ and $-\widetilde\Delta$, the right-hand side is bounded from below if $\Delta$ and $-\widetilde\Delta$ would be as small as possible under the assumption that $a_0=a$ and $a_1=1$. This is obtained when we exchange $\Delta$ by 
\[
\Delta'=[a,1,\overline{1,a}]= a + \frac{2-a-\sqrt{a^2+4a}}{2-4a}
\]
and $-\widetilde\Delta$ by
\[
-\widetilde\Delta'=[0,\overline{a,1}]= -\frac{1}{2}+\frac{\sqrt{a^2+4a}}{2a}.
\]
With these values we obtain
\begin{align*}
    \max_{0\leq j\leq a} \frac{\vol(B_j)}{\det(\Lambda)} 
    &\ge \frac{(1-\frac{a}{2}+\Delta')(1+\frac{a}{2}-\widetilde\Delta')}{\Delta'-\widetilde\Delta'}\\
    &=\frac{(8a^6+18a^5-25a^4-57a^3-32a^2+64a-16)+(6a^4-5a^3-31a^2+8a+4)\sqrt{a^2+4a}}{8(4a^5-6a^4-2a^3-6a^2+6a-1)}\\
    &>1+\frac{(a^2+8)\sqrt{a^2+4}}{4(a^2+4)}=\disp([\overline{a}]),
\end{align*}
where you can check the last inequality by considering their difference and see that it has no zero for any $a\ge 1$.
For odd $a$ we have the lower bound
\begin{align*}
    \max_{0\leq j\leq a} \frac{\vol(B_j)}{\det(\Lambda)} 
    &\ge \frac{(1-\frac{a}{2}-\frac{1}{2}+\Delta)(1+\frac{a}{2}+\frac{1}{2}-\widetilde\Delta)}{\Delta-\widetilde\Delta},
\end{align*}
and we choose the same $\Delta'$ and $-\widetilde\Delta'$, arriving at
\begin{align*}
    \max_{0\leq j\leq a} \frac{\vol(B_j)}{\det(\Lambda)} 
    &\ge \frac{(1-\frac{a}{2}-\frac{1}{2}+\Delta')(1+\frac{a}{2}+\frac{1}{2}-\widetilde\Delta')}{\Delta'-\widetilde\Delta'}\\
    &=\frac{(16a^7+36a^6-42a^5-130a^4-38a^3+102a^2-24a)+(12a^5-18a^4-54a^3+26a^2+2a)\sqrt{a^2+4a}}{16(4a^6-6a^5-2a^4-6a^3+6a^2-a)}\\
    &>1+\frac{(a^2+7)\sqrt{a^2+4}}{4(a^2+4)}=\disp([\overline{a}]),
\end{align*}
where you can check the last inequality again as before by considering their difference and see that it has no zero for any $a\ge 1$.

With this we can now assume that for all $i\in\Z$ that $a_i=a$ implies $a_{i+1} \neq 1$ and $a_i=1$ implies $a_{i+1} \neq a$.
Additionally assume $1<a_1<a$. The right-hand side can be bounded from below if we exchange $\Delta$ and $-\widetilde\Delta$ by
\[
\Delta'=[a,a-1,\overline{1,a-1}]=  a- \frac{1}{2} +\frac{\sqrt{(a-1)(a+3)}}{2a-2}
\quad \text{and} \quad
-\widetilde\Delta'=[0,\overline{a,2}]= -1 + \frac{\sqrt{a^2+2a}}{a}.
\]
With these choices we can show the same bound as above, but we omit the calculation since these are a bit more unpleasant than the calculations above, but work in a similar way.
\end{proof}

We end this section with a short example that is obvious for large $a$ and can be easily verified with the computer for small $a$.

\begin{ex}
As $a\to \infty$ the lattice generated by $\Delta=[a,\overline{a-1,1}]$ and $\widetilde\Delta=[0,\overline{a-1,1}]$ has dispersion equal to the upper bound of the previous theorem applied to $a-1$ which happens as one shifts the $a$ further to infinity. This shows that the biggest maximal empty box does not necessarily appear at the largest coefficient and that there is not always a largest empty box, i.e. justifying the supremum in the definition. 
\end{ex}

%%%%%%%%%%%%%%%%%%%%%%%%%%%%%%%%%%%%%%%%%%%%%%%%%%%%%%%%%%%%%%%%%%%%%%%%%%%%

\section{Lattices with small dispersion}\label{sec:small dispersion}

It follows directly from Theorem \ref{thm:Thight Bound} that the quadratic lattices generated by $\varphi$ and $1+\sqrt{2}$, of all lattices, have the smallest and second smallest normalized dispersion. In this section we find the lattices with $n$-th smallest normalized dispersion. These turn out to be generated by quadratic rationals rather than quadratic integers.

\begin{thm}\label{thm:BestLattices}
The lattice with the smallest normalized dispersion is $\Lambda_5$ generated by 
\[
\begin{pmatrix}
1   &\varphi\\
1   &\overline \varphi
\end{pmatrix}
\quad\text{and satisfies}\quad
\frac{\disp(\Lambda_5)}{\det(\Lambda_5)}=\frac{\varphi^3}{\sqrt{5}}=1.89442\dots
\]
The lattice with the second smallest normalized dispersion is $\Lambda_2$ generated by
\[
\begin{pmatrix}
1   &1+\sqrt{2}\\
1   &1-\sqrt{2}
\end{pmatrix}
\quad\text{and satisfies}\quad
\frac{\disp(\Lambda_2)}{\det(\Lambda_2)}=\frac{3+2\sqrt{2}}{2\sqrt{2}}=2.06066\dots
\]
The lattice with the $(n+1)$-th smallest normlized dispersion is generated by
\[
\begin{pmatrix}
1   &-\Delta \\
1   &-\overline\Delta
\end{pmatrix}
\quad\text{with}\quad
\Delta = [\overline{2,\underbrace{1,\dots,1}_{2n},2}].
% \quad\text{and}\quad
% -\overline\Delta = [0,\overline{2,\underbrace{1,\dots,1}_{2n},2}],
\]
The sequence consisting of the $n$-th smallest normalized dispersion converges to
\[
\frac{4+5^{1/2}}{3}=2.07868\dots
\]
which is the normalized dispersion of the lattice generated by
\[
\begin{pmatrix}
1   &-\Delta \\
1   &-\widetilde\Delta
\end{pmatrix}
\quad\text{with}\quad
\Delta =[2,\overline 1] =\varphi^2 \quad\text{and}\quad -\widetilde\Delta=[0,2,\overline1] = \varphi^{-2}.
\]
\end{thm}

% \begin{remark}
% The sequence of $n$-th smallest dispersion converges exponentially at a rate of $\varphi^{-4n}$.\jaspar{I am not totally sure that I have this right.} The formula at the beginning of page 19 (which I am guessing is close to the formula for the actual dispersion of the nth smallest lattice) we obtain
% \[
% \frac{2.07868\dots}{1+\varphi^{-4n}0.198632+\varphi^{-8n}0.011336}+\frac{\varphi^{-4n}0.42524+\varphi^{-8n}0.00671024}{1+\varphi^{-4n}0.198632+\varphi^{-8n}0.011336}
% \]
% \[
% \frac{2.07868\dots}{1+\varphi^{-4n}0.198632+\varphi^{-8n}0.011336}+\frac{0.42524+\varphi^{-4n}0.00671024}{\varphi^{4n}+0.198632+\varphi^{-4n}0.011336}
% \]
% \end{remark}
\begin{remark}
Note that the $n$-th best lattice is generated by
\begin{align*}
\Delta    &= \frac{(F_{2n+4}+F_{2n+2})+\sqrt{9F_{2n+3}^2-4}}{2F_{2n+3}} \quad \text{and} \quad 
\widetilde\Delta    = \overline\Delta= -\Delta^{-1}=\frac{(F_{2n+4}+F_{2n+2})-\sqrt{9F_{2n+3}^2-4}}{2F_{2n+3}} 
\end{align*}
with normalized dispersion
\[
1+\frac{2F_{2n+4}}{\sqrt{9F^2_{2n+3}-4}}.
\]
\begin{remark}\label{rem:Lagrange Connection} Another way to write the generator of the $n$-th best lattice is
\[
\Delta= 1+ \frac{1}{2}\vast(\frac{F_{2n}}{F_{2n+3}}+\sqrt{9-\frac{4}{F_{2n+3}^2}}\vast),
\]
\end{remark}
which features the Lagrange Numbers corresponding to the Markov triples of the form $(1,y,z)$ , i.e. $1+y^2+z^2=3yz$. This suggest that there is some connection between Lagrange and Markov spectra and dispersion, which the authors did not explore any further.
\end{remark}
We start with a lemma that will come in handy later on.
\begin{lemma}\label{lem:FactorChangingDispersion}
For $2<\Delta<3$ we have
\[
\frac{\lambda\Delta\Big(2+\frac{1}{\lambda\Delta}\Big)}{\lambda\Delta+\frac{1}{\lambda\Delta}} < \frac{\Delta\Big(2+\frac{1}{\Delta}\Big)}{\Delta+\frac{1}{\Delta}},
\]
if $0<\lambda <1$, and
\[
\frac{\lambda\Delta\Big(2+\frac{1}{\lambda\Delta}\Big)}{\lambda\Delta+\frac{1}{\lambda\Delta}} > \frac{\Delta\Big(2+\frac{1}{\Delta}\Big)}{\Delta+\frac{1}{\Delta}},
\]
if $1<\lambda<2$.
In particular, if a lattice has generators $\Delta$ and $\widetilde\Delta=-\Delta^{-1}$, then increasing $\Delta$ has the effect of increasing the volume $\Vol(B_1)$, while decreasing $\Delta$ causes the opposite effect.
\end{lemma}

\begin{proof}
We begin by subtracting the right hand side from the left hand side and finding the common denominator. Since the denominator of the resulting fraction is positive, the sign of the difference, and therefore, the inequality in the statement of the lemma, depends only on the numerator:
\[
\Delta\bigg(\Big(2\Delta - \Delta^2\Big)\lambda + 1+2\Delta\bigg)(\lambda-1)=\Delta\Big(2\Delta - \Delta^2\Big)\Bigg(\lambda - \frac{1+2\Delta}{\Delta^2 - 2\Delta}\Bigg)(\lambda-1).
\]
The assumption $2<\Delta<3$ and $\lambda<2$ implies
\[
2\Delta-\Delta^2<0 \quad \text{and} \quad \lambda-\frac{1+2\Delta}{\Delta^2 - 2\Delta}<\lambda-2<0.
\] 
Thus we see that the sign of the numerator is depends only on $\lambda-1$.
\end{proof}

We are going break up the proof of Theorem \ref{thm:BestLattices} into several lemmas to make it more readable and to simplify notation we will use the correspondence between lattices and sequences, i.e.\ the sequence $(a_i)_{i\in\Z}$ represents the lattice generated by $\Delta=[a_0,a_1,\dots]$ and $-\widetilde\Delta=[0,a_{-1},a_{-2},\dots]$. 

\begin{lemma}
If for $\mathfrak{A}=(a_i)_{i\in\Z}$ there exists an $i\in\Z$ such that $a_i\ge 3$, then
\[
\disp(\mathfrak{A}) \ge 1 + \frac{4}{\sqrt{13}}=2.10940\dots
\]
In particular, if a sequence is to have dispersion smaller than the limit in Theorem \ref{thm:BestLattices}, it cannot contain an element larger than $2$.
\end{lemma}

\begin{proof}
This follows directly by Theorem \ref{thm:Thight Bound}.
\end{proof}

The following lemma is a well-known, useful tool for manipulating continued fractions’ coefficients we will use heavily without reference in the upcoming lemmas. 
\begin{lemma}\label{lem:We are dumb}
Increasing a continued fraction coefficient with an even index makes the number bigger while increasing a continued fraction coefficient with an odd index makes the number smaller, similarly decreasing a continued fraction coefficient with an even index makes the number smaller while decreasing a continued fraction coefficient with an odd index makes the number bigger.
\end{lemma}

In the next seven lemmas we are going to prove that the only lattices that can have dispersion less than $\disp(\overline{1},2,2,\overline{1})$ are those claimed in the theorem. Each lemma shows that if a sequence consisting of $1$'s and $2$'s contains a particular pattern, then its dispersion will be too big.  Our strategy is to apply Corollary \ref{cor:sequence correspondence} to re-index the sequence so that the pattern to be eliminated is roughly centered on $a_0$ and then apply Lemma \ref{lem:Normalized Volume of a Box} obtain a bound on $\vol(B_1)$ by minimizing $\Delta$ and $\widetilde\Delta$ under the assumption that they do not contain a pattern that we have previously demonstrated leads to a dispersion bigger than we want.

% More precisely the main method will be to assume that a certain component is part of the sequence and then giving some values $\Delta'$ and $-\widetilde\Delta'$ such that

% \[
% \frac{(\Delta)(2-\widetilde\Delta)}{\Delta-\widetilde\Delta} >\frac{(\Delta')(2-\widetilde\Delta')}{\Delta'-\widetilde\Delta'} >   \disp(\overline{1},2,2,\overline{1}) .
% \]

\begin{lemma}
Assume that $\mathfrak{A}=(a_i)_{i\in\Z}$ is a $(1,2)$-sequence. If the sequence contains the pattern $(2,1,2)$, then 
\[
\disp(\mathfrak{A}) \ge \frac{1}{23}\Big(41+4\sqrt{3}\Big)=2.08383\dots
\]
In particular, if a sequence is to have dispersion smaller than the limit in Theorem \ref{thm:BestLattices}, it cannot contain any isolated $1$'s.
\end{lemma}

\begin{proof}
By Corollary \ref{cor:sequence correspondence}, we may re-index the sequence so that $a_{-2}=2,a_{-1}=1,a_{0}=2$. With the restriction on $\mathfrak{A}$, the minimum possible values for $\Delta$ and $\widetilde\Delta$ are
\begin{align*}
\Delta'=[2,\overline{2,1}] \quad \text{and} \quad -\widetilde\Delta'=[0,1,2,\overline{2,1}].   
\end{align*}
Applying Lemma \ref{lem:Normalized Volume of a Box} we estimate
\[
   \frac{\vol(B_1)}{\det(\Lambda)} 
    \ge \frac{(\Delta')(2-\widetilde\Delta')}{\Delta'-\widetilde\Delta'} =\frac{1}{23}(41+4\sqrt{3})=2.08383\ldots >\disp(\overline{1},2,2,\overline{1}). \qedhere
\]
\end{proof}

\begin{lemma}
Assume that $\mathfrak{A}=(a_i)_{i\in\Z}$ is a $(1,2)$-sequence and contains no isolated $1$'s. If the sequence contains the pattern $(1,2,1)$, then 
\[
\disp(\mathfrak{A}) \ge \frac{4}{15}\Big(3+2\sqrt{6}\Big)=2.10639\dots
\]
In particular, if a sequence is to have dispersion smaller than the limit in Theorem \ref{thm:BestLattices}, it cannot contain any isolated $2$'s.
\end{lemma}

\begin{proof}
Again by index shifting we can assume $a_{-1}=1,a_{0}=2,a_{1}=1$ and by the assumptions in the statement we must have $a_{-2}=1,a_{2}=1$. Under the constrains of the lemma, the minimum possible values for $\Delta$ and $\widetilde\Delta$ are
\begin{align*}
\Delta'=[2,1,1,\overline{2,1,1,1}] \quad \text{and} \quad -\widetilde\Delta'=[0,1,1,\overline{2,1,1,1}].
\end{align*}
Applying Lemma \ref{lem:Normalized Volume of a Box} we estimate
\[
    \frac{\vol(B_1)}{\det(\Lambda)} 
   \ge \frac{(\Delta')(2-\widetilde\Delta')}{\Delta'-\widetilde\Delta'} =\frac{4}{15}(3+2\sqrt{6})=2.10639\dots \qedhere
\]
\end{proof}

\begin{lemma}
Assume that $\mathfrak{A}=(a_i)_{i\in\Z}$ is a $(1,2)$-sequence and contains neither isolated $1$'s nor $2$'s. If the sequence contains contains at least three consecutive $2$'s, then 
\[
\disp(\mathfrak{A}) \ge \frac{2}{2519}(1157+101\sqrt{210})=2.08068\dots
\]
In particular, if a sequence is to have dispersion smaller than the limit in Theorem \ref{thm:BestLattices}, then the $2$'s appear in isolated pairs.
\end{lemma}

\begin{proof}
With $a_{-2}=2,a_{-1}=2,a_{0}=2,a_{1}=a_{1}=1$, the smallest that $\Delta$ and $\widetilde\Delta$ can be given the assumptions on $\mathfrak{A}$ are
\begin{align*}
\Delta'=[2,1,1,\overline{2,2,2,1,1,1}] \quad \text{and} \quad -\widetilde\Delta'=[0,2,2,2,\overline{1,1,1,2,2,2}],
\end{align*}
yielding
\[
    \frac{\vol(B_1)}{\det(\Lambda)} 
    \ge \frac{(\Delta')(2-\widetilde\Delta')}{\Delta'-\widetilde\Delta'} =\frac{2}{2519}(1157+101\sqrt{210})=2.08068\dots \qedhere
\]
\end{proof}

\begin{lemma}
Assume that $\mathfrak{A}=(a_i)_{i\in\Z}$ is a $(1,2)$-sequence, has no isolated $1$'s, and the $2$'s only occur in pairs separated by $1$'s. If the lengths of the stretches of $1$'s that separate the $2$'s do not all have the same parity, then  
\[
    \disp(\mathfrak{A}) \ge \frac{4+5^{1/2}}{3}.
\]
\end{lemma}

\begin{proof}
Without loss of generality $a_{-2}=1,a_{-1}=2,a_{0}=2,a_{1}=a_{2}=\dots=a_{2k+1}=1$ for some $k\in \N$ and that the first stretch of $1$'s in negative direction is even, i.e. $a_{-2}=a_{-3}=\dots=a_{-2n+1}=1$ and $a_{-2n}=2$. The values
\begin{align*}
\Delta'=[2,\overline 1] \quad \text{and} \quad -\widetilde\Delta'=[0,2,\overline 1] 
\end{align*}
are as small as possible under the constraints in the statement resulting in
\[
   \frac{\vol(B_1)}{\det(\Lambda)} 
   \ge \frac{(\Delta')(2-\widetilde\Delta')}{\Delta'-\widetilde\Delta'} =\frac{4+5^{1/2}}{3}
    =\disp(2,2,\overline{1})>\disp(\overline{2,2,1,1}). \qedhere
\] 
\end{proof}

\begin{lemma}\label{lem:oneallodd}
Assume that $\mathfrak{A}=(a_i)_{i\in\Z}$ is a $(1,2)$-sequence has no isolated $1$'s, the $2$'s only occur in pairs separated by $1$'s, and the lengths of the stretches of $1$'s that separate the $2$'s all have the same parity. If the lengths of the stretches of $1$'s that separate the $2$'s are all odd, then  
\[
    \disp(\mathfrak{A}) \ge \frac{4+5^{1/2}}{3}.
\]
\end{lemma}

\begin{proof}
By transforming the sequence, if necessary, we may assume that $\Delta$ and $\widetilde \Delta$ are of the form
\begin{align*}
\Delta=[2,\underbrace{1,\dots,1}_{2k+1},2,2,\Delta_{2k+4}] \quad \text{and} \quad -\widetilde\Delta=[0,2,\underbrace{1,\dots,1}_{2l+1},2,2,-\widetilde\Delta_{-2l-5}],
\end{align*}
with $l>k$.
Then we have that $\vol(B_{-1})<\vol(B_{1})$. Now
\[
\Delta'=[2,\underbrace{1,\dots,1}_{2l+1},2,2,X]<\Delta,
\]
with $X$ being an arbitrary $(1,2)$-sequence. If $\mathfrak{A}'$ is the sequence obtained by replacing $\Delta$ by $\Delta'=-\widetilde\Delta^{-1}$ then we have
\[
\vol(B_{-1})<\vol(B'_{-1})=\vol(B'_{1})<\vol(B_{1}),
\]
with the equality following from the symmetry of the sequence $\mathfrak{A}'$.

Now we want to simultaneously extend the number of initial $1$'s in both the negative and positive direction of the sequence $\mathfrak{A}'$.
Extending the initial stretch of $1$'s in the positive direction has the effect of making $\Delta'$ smaller, while doing the same in the negative direction makes $\widetilde\Delta=-\Delta^{-1}$ larger. By Lemma \ref{lem:FactorChangingDispersion} we see that the former dominates the latter. From this it follows that the smallest possible dispersion under the constraints converges to the dispersion associated to
\[
\Delta'=[2,\overline{1}].
\]
If all the stretches of $1$'s are of same length we see that the assertion follows from the last bit we just did above.
\end{proof}

\begin{lemma}
Assume that $\mathfrak{A}=(a_i)_{i\in\Z}$ is a $(1,2)$-sequence, has no isolated $1$'s, the $2$'s only occur in pairs separated by $1$'s, and the lengths of the stretches of $1$'s that separate the $2$'s are all even. If the sequence is eventually all $1$'s, then
\[
    \disp(\mathfrak{A}) \ge \frac{4+5^{1/2}}{3}.
\]
\end{lemma}

\begin{proof}
Without loss of generality assume $a_{-1}=a_{0}=2,a_{k}=1$ for all $k\in\N$, then
\begin{align*}
\Delta=[2,\overline{1}] \quad \text{and} \quad -\widetilde\Delta \ge [0,2,\overline{1}], 
\end{align*}
and we see that in this case the dispersion is bigger than the limit in Theorem \ref{thm:BestLattices}.
\end{proof}

\begin{lemma}
Assume that $\mathfrak{A}=(a_i)_{i\in\Z}$ is a $(1,2)$-sequence, has no isolated $1$'s, the $2$'s only occur in pairs separated by $1$'s, and the lengths of the stretches of $1$'s that separate the $2$'s are all even, and the sequence is not eventually all $1$'s. If the lengths of the stretches of $1$'s that separate the $2$'s are not all the same, then  
\[
    \disp(\mathfrak{A}) \ge \frac{4+5^{1/2}}{3}.
\]
\end{lemma}

\begin{proof}
By transforming the sequence, if necessary, we may assume that $\Delta$ and $\widetilde \Delta$ are of the form 
\begin{align*}
\Delta=[2,\underbrace{1,\dots,1}_{2k},2,2,\Delta_{2k+3}] \quad \text{and} \quad -\widetilde\Delta=[0,2,\underbrace{1,\dots,1}_{2l},2,2,-\widetilde\Delta_{-2l-4}],
\end{align*}
with $l<k$.
For fixed $k$ we get that $\Delta$ and $-\widetilde\Delta$ are bounded from below by
\[
\Delta'=[2,\underbrace{{1,\dots,1}}_{2k},2,\overline{1}] \quad \text{and} \quad -\widetilde\Delta'=[0,2,\underbrace{{1,\dots,1}}_{2k-2},2,2,\overline{1}].
\]
The remainder of the proof is similar to the proof of Lemma \ref{lem:oneallodd} in that we simultaneously add two additional $1$'s to the initial stretch of $1$'s in $\Delta'$ and $-\widetilde\Delta'$, making $\Vol(B'_1)$ smaller. Repeating this process leads to the limit in Theorem \ref{thm:BestLattices}. 

Applying Lemma \ref{lem: Delta i} we can rewrite $\Delta'$ and $-\widetilde\Delta'$ as
\[
\Delta'=[2,\underbrace{{1,\dots,1}}_{2k},2,\varphi]=\frac{F_{2k+5}\varphi+F_{2k+3}}{F_{2k+3}\varphi+F_{2k+1}}=:\frac{A(k)}{A(k-1)} \quad\text{and }
\]
\[
-\widetilde\Delta'=[0,2,\underbrace{{1,\dots,1}}_{2k-2},2,\varphi^2]=\frac{F_{2k+1}\varphi^2+F_{2k-1}}{F_{2k+3}\varphi^2+F_{2k+1}}=:\frac{B(k-1)}{B(k)}.
\]
Make the subsitution $F_k=(\varphi^k-\overline\varphi^{k})/\sqrt{5}$
and evaluate the formula in Lemma \ref{lem:Normalized Volume of a Box} for $\vol(B_1)$ to obtain
\begin{align*}
\frac{(\Delta')(2-\widetilde\Delta')}{\Delta'-\widetilde\Delta'} &= \frac{2A(k)B(k)+A(k)B(k-1)}{A(k)B(k)+A(k-1)B(k-1)} \\
        &= \frac{\varphi^{4k}(\varphi^{13}+\varphi^{9}+2\varphi^{7}+2\varphi^{4}+\varphi^{2})+4\varphi^{7}+3\varphi^{4}+2\varphi^{-1}+\varphi^{-3}+\varphi^{-4k}(1+5\varphi^{-3})}{\varphi^{4k}(\varphi^{12}+1)+2\varphi^{7}+4\varphi+\varphi^{-4k}(2\varphi+2\varphi^{-3})}\\
        &= \frac{\varphi^{4k}(300\varphi+186)+ (65\varphi+33) +\varphi^{-4k}(10\varphi - 14)}{\varphi^{4k}(144\varphi+90)+(30\varphi+16)+\varphi^{-4k}(6\varphi - 6)}.
\end{align*}
The numerator of the derivative of the last term with respect to $k$ is
\begin{align*}
&(4k\varphi^{4k-1}(300\varphi+186)-4k\varphi^{-4k-1}(10\varphi - 14))(\varphi^{4k}(144\varphi+90)+(30\varphi+16)+\varphi^{-4k}(6\varphi - 6))\\
-&(\varphi^{4k}(300\varphi+186)+ (65\varphi+33) +\varphi^{-4k}(10\varphi - 14))(4k\varphi^{4k-1}(144\varphi+90)-4k\varphi^{-4k-1}(6\varphi - 6)).
\end{align*}

By showing that this is negative we will see that increasing $k$ decreases $\vol(B_1)$.
Dividing the last term by $4k$, $\varphi^{-4k-1}$  and $\varphi^{-4k}$ does not change the sign, so we can consider the simplified term
\begin{align*}
&(\varphi^{8k}(300\varphi+186)-(10\varphi - 14))(\varphi^{8k}(144\varphi+90)+\varphi^{4k}(30\varphi+16)+(6\varphi - 6))\\
-&(\varphi^{8k}(300\varphi+186)+ \varphi^{4k}(65\varphi+33) +(10\varphi - 14))(\varphi^{8k}(144\varphi+90)-(6\varphi - 6))\\
=&\varphi^{4k}(158\varphi+116)+\varphi^{8k}(1584\varphi+1008)-\varphi^{12k}(582\varphi+354),
\end{align*}
which is negative for $k>1$. 
\end{proof}

We have now eliminated the possibility that any two-sided sequence other than those proposed in Theorem \ref{thm:BestLattices} can have dispersion smaller than the limit. We finish the proof of the theorem with the following lemma.

\begin{lemma}
The sequence $(\disp(\overline{2,\underbrace{1,\dots,1}_{2n},2}))_n$ is increasing with limit
\[
\frac{4+5^{1/2}}{3}=2.07868\dots
\]
\end{lemma}

\begin{proof}
This follows directly from Lemma \ref{lem:FactorChangingDispersion}.
\end{proof}

%%%%%%%%%%%%%%%%%%%%%%%%%%%%%%%%%%%%%%%%%%%%%%%%%%%%%%%%%%%%%%%%%%%%%%%%%%%%%%%%%%
\section{Periodic Disperson}\label{sec:Periodic Disperson}

A \emph{rational lattice} in the unit square is a point set of the form $P_n=\{(\{k\alpha\},k/n),0\leq k<n\}\subseteq[0,1]^2$, where $0<\alpha<1$ and $\{k\alpha\}$ denotes the decimal part of $k\alpha$. In the special case $\alpha=p/n$ with $p$ and $n$ relatively prime, we call $P_n$ a \emph{rank-1 lattice}. Such a point set can be realized as a subset of the lattice $\Lambda\subset\R^2$ with matrix representation
\[
\begin{pmatrix}
\alpha      &   -1\\
n^{-1}  &   0
\end{pmatrix}.
\]
These point sets where studied in \cite{BreneisHinrichs} within the context of periodic dispersion. In \cite{Ullrich} it was shown that $d$ is a lower bound for the normalized periodic dispersion in the $d$-dimensional case. There authors showed that Fibonacci lattices defined by
\begin{align}\label{eq:Fibonacci Lattice}
\mathcal{F}_m=\Bigg\{\Bigg(\bigg\{\frac{k F_{m-2}}{F_m}\bigg\},\frac{k}{F_m}\Bigg),0\leq k<F_m\Bigg\},
\end{align}
have optimal normalized periodic dispersion of $2$ for all $m\in\N$ and moreover, they are the only rank-1 lattices that achieve this. In this section we will modify our framework to give an alternative proof of this result. 
We will also be able to draw a connection to Zaremba's long-standing conjecture. This conjecture is already known to be connected to $L^2$-discrpancy, which is another measure for the uniformity of points, see for example \cite{Larcher}.

\begin{conjecture}[Zaremba]\label{con:Zaremba}
There exists an absolute constant $A$ such that for all $N\in\N$ there is an $a\in\N$ relatively prime to $N$ such that all continued fraction coefficients of $\frac{a}{N}$ are less than $A$. (The choice $A=5$ is compatible with the numerical evidence.)
\end{conjecture}

% A228804: The numbers which are not a denominator of a rational number having continued fraction consisting entirely of 1s and 2s, the complement of A228803. 6, 9, 14, 20, 22, 23, 28, 32,\dots 
% This sequence is in some sense bad. (See \cite{oeis})

% \begin{conjecture}[Moser]\label{con:Moser}
% There is an absolute constant $B$ such that for all $N\in\N$ there is an $a\in\N$ relatively prime to $N$ such that the sum of all continued fraction coefficients of $\frac{a}{N}$ is less than $B\log(N)$.
% \end{conjecture}

When we normalize the matrix representing a rational lattice in order to fit it into the framework of Section \ref{sec:contfracconnect} we obtain
\[
\begin{pmatrix}
1   &-\Delta\\
1   &-\widetilde\Delta
\end{pmatrix}
=
\begin{pmatrix}
\alpha^{-1}   &   0\\
0   &   n
\end{pmatrix}
\begin{pmatrix}
\alpha      &   -1\\
n^{-1}  &   0
\end{pmatrix}
=
\begin{pmatrix}
1   &   -\alpha^{-1}\\
1   &   0
\end{pmatrix}.
\]
Unfortunately, the lattice generated by this matrix certainly has multiple points on horizontal lines and if $\alpha$ is rational we have multiple points on vertical lines as well. Therefore, without modifying our framework, we are unable to handle these lattices.

 To simplify the matter we a going to assume that the matrix that generates the lattice already has the form that we want. Thus, we do not have to worry about modifying Lemma \ref{lem:EquivalentLattice}. Lemma \ref{lem:StartingBoxIsEmpty} already fits in the modified framework and gives us a starting box for Proposition \ref{prop:BoxesAreSemiConvergents}.
Once we modify Proposition \ref{prop:BoxesAreSemiConvergents} in such a way as to allow the sequences associated to lattices to terminate in one or both directions, the remainder of section \ref{sec:contfracconnect} will follow without any significant differences, in particular Theorem \ref{thm:Genreal upper/lower bounds} holds.

Note that for a finite continued fraction we have $[a_0,a_1,\dots,a_k+1]=[a_0,a_1,\dots,a_k,1]$.
By excluding the sequence consisting of a single $1$, i.e. excluding the trivial case $\Lambda = \Z^2$, we may always assume that the sequence terminates in $1$ and, by applying Lemma \ref{lem:Can start anywhere},  that $\widetilde\Delta<0$.

\begin{prop}[Modification of Proposition \ref{prop:BoxesAreSemiConvergents}]
Let $\Lambda$ be the lattice generated by $(1,1)$ and $(-\Delta,-\widetilde\Delta)$, where $1\leq\Delta$ and $-1\le\widetilde\Delta<0$ are possibly rational. For all $n\in[A_{l-1},A_{k+1}]$ such that $B_n\in \mathcal{B}$ write $n=A_i+j$, where $0\leq j<a_i$, then the left and right sides of the $n$-th box in $\mathcal B$ are given by
\begin{align*}
    \alpha_n&=\begin{cases}
     p_{i}-q_{i}\Delta  &\text{if } i \text{ is odd}, \\
    j(p_{i}-q_{i}\Delta) + (p_{i-1}-q_{i-1}\Delta)  &\text{if } i \text{ is even},
    \end{cases}\\
    \beta_n&=\begin{cases}
    j(p_{i}-q_{i}\Delta) + (p_{i-1}-q_{i-1}\Delta)  &\text{if } i \text{ is odd},\\
    p_{i}-q_{i}\Delta  &\text{if } i \text{ is even}.
    \end{cases}
\end{align*}
% In particular, if the sequence ends in the positive direction, then the point bounding box $B_{A_{k+1}}$(the tallest maximal empty box bounded by the origin from below) will have $0$ in its first coordinate and if the sequence ends in the negative direction, then both points bounding the box $B_{A_{l-1}}$(the shortest maximal empty box bounded by the origin from below) will have $[0, a_{-1},\dots,a_l]$ in their second coordinate, which is the minimal possible positive second coordinate of this lattice.
\end{prop}

\begin{proof}
By the previous lemma $B_0$ is a maximal empty box and trivially satisfies the proposition.
Following the proof of Proposition \ref{prop:BoxesAreSemiConvergents}, if $B_{n}$ is a maximal empty box and satisfies the assertion $B_{n+1}$ will satisfy it as well as long as $\alpha_n + \beta_n \neq 0$. This is the case until the continued fraction algorithm terminates which occurs exactly at $n=A_{k+1}$. Again, not unlike the second part of the proof of Proposition \ref{prop:BoxesAreSemiConvergents}, for negative $n$, we can convert the up statement into a down statement. The proof for negative $n$ is then a facsimile to the positive case.
\end{proof}

With our machinery we are able to give an easy proof of the main theorems in \cite{BreneisHinrichs}.

\begin{thm}\label{thm:Fibonacci lattices are optimal}
Let $m\ge 3$ be an integer. The Fibonacci lattice $\mathcal{F}_m$ satisfies
\[
\disp(\mathcal{F}_m)=\frac{2}{F_m}.
\]
Furthermore, no other rank-1 lattice achieves this bound.
\end{thm}

\begin{proof}
The Fibonacci lattice corresponds to the planar lattice generated by
\begin{equation}
\begin{pmatrix}
1   &-\frac{F_{m-1}}{F_{m-2}}\\
1   &1
\end{pmatrix},
\end{equation}
where $\frac{F_{m-1}}{F_{m-2}}=[\underbrace{1,\dots,1}_{m-2}]$.
For this we can calculate all maximal empty boxes to be for $0\leq k\leq m-3$
\begin{align*}
\vol(B_k)&= \bigg(1+\frac{F_{m-k-1}}{F_{m-k-2}} \bigg) \bigg(1+\frac{F_{k+1}}{F_{k+2}} \bigg) \bigg(\frac{F_{m-k-1}}{F_{m-k-2}}+\frac{F_{k+1}}{F_{k+2}} \bigg)^{-1}\\
&= \frac{F_{m-k}F_{k+3}}{F_{m-k-1}F_{k+2}+F_{m-k-2}F_{k+1}}\\
&= \frac{F_{m-k}F_{k+3}}{F_{m}},
\end{align*}
having its maximum at $k=0$ and $k=m-3$.
By Theorem \ref{thm:Genreal upper/lower bounds} we directly see that any other rank-1 lattice has to have a higher dispersion as the continued fractions coefficients of the associated generators are larger than 1. (see Table \ref{tab:dispbyai} in the Appendix)
\end{proof}

\begin{remark}
The Fibonacci lattice can be realized by taking the points from lattice generated by the matrix 
\[
\frac{1}{F_m}\begin{pmatrix}
F_{m-2}     &   -F_m\\
1  &   0
\end{pmatrix}
\]
that are in the unit square.  The authors of \cite{BreneisHinrichs} also observed that by scaling and translating the points from that lattice contained in the region $[-1,1)\times[0,1)$ one can also achieve optimal periodic dispersion and asked if there there were other lattice point sets that are optimal. Our framework can be used to show that the only other lattice point set (up to transformation by the symmetries of the square) that achieves optimal periodic dispersion is the one obtained by scaling and translating the part of the lattice contained in $[-1,1)^2$. This point set was studied in the context of battleships in \cite{Fiat1989HowTF}. 
\end{remark}

The connection to Zaremba's Conjecture will require the following lemma.

\begin{lemma}\label{lem:FindTorusPointSet}
Let $p/n=[0,a_1,\dots,a_k,1]$ be a finite continued fraction then there is a rank-$1$ lattice of size $n$ having the same periodic dispersion equal to the size of the largest maximal empty box in the lattice generated by
\[
\begin{pmatrix}
1     &   -\frac{n}{p}\\
1  &   0
\end{pmatrix},
\]
normalized by the determinant divided by $n$.
\end{lemma}

\begin{proof}
Consider the dispersion equivalent matrix
\[
\begin{pmatrix}
\frac{p}{n}     &   -1\vspace{0.2em}\\ 
\frac{1}{n}  &   0
\end{pmatrix}
=
\begin{pmatrix}
\frac{p}{n}     &   0\vspace{0.2em}\\
0 &   \frac{1}{n}
\end{pmatrix}
\begin{pmatrix}
1     &   -\frac{n}{p}\vspace{0.2em}\\
1  &   0
\end{pmatrix},
\]
call this lattice $\Lambda$ and draw the box $[0,1)^2$. So the points inside this square are equal to $P_n$.
Since $(0,0)$, $(1,0)$, $(0,1)$, and $(1,1)$ are in $\Lambda$ we see that
\[
\Lambda=P_n + \Z^2.
\]
Now, we see that every maximal box in $\Lambda$ directly corresponds to a periodic maximal box in $[0,1)^2$ amongst the points $P_n$. Finally, infinite strips in the lattice do correspond to boxes of size $1/n$ in $[0,1)^2$ which are essentially not maximal there, since they can directly be extended to maximal boxes of size $2/n$.
\end{proof}

\begin{thm}\label{thm:Zaremba eq disp Bound}
Zaremba's Conjecture is true if and only if there is a positive constant $C$ such that for all $n\in\N$ there exists some $p\in\N$, co-prime to $n$, such that the rank-$1$ lattice $P_n=\{(\{kp/n\},k/n),0\leq k<n\}\subseteq[0,1]^2$ has
\[
 \disp_\mathbb{T}(P_n) < \frac{C}{n}.
\]
\end{thm}

% \begin{thm}\label{thm:Zaremba eq disp Bound}
% Conjecture \ref{con:Zaremba} is equivalent to Conjecture \ref{con:TorusDispBound}.
% \end{thm}

\begin{proof}
 Assume Zaremba's Conjecture is true with some $A\in\N$. For all $n\in\N$ then there is a $p\in\N$ such that $p/n=[0,a_1,\dots,a_k,1]$ with $a_i\leq A$. By the upper bound of Theorem \ref{thm:Genreal upper/lower bounds} we get that the biggest maximal empty box of the lattice $\Lambda$ generated by 
\[
\begin{pmatrix}
1     &   -\frac{n}{p}\\
1  &   0
\end{pmatrix},
\]
has area smaller than 
\[
C:=\frac{A}{4}+\frac{3}{2}+\frac{1}{A+2}.
\]
Thus, by Lemma \ref{lem:FindTorusPointSet}, we know that the corresponding rank-$1$ lattice $P_n$ fulfills 
\[
\disp_\mathbb{T}(P_n)<\frac{C}{n}.
\]

Now assume that there exists some absolute constant $C$ such that for every $n\in\N$ there exists some $p\in\N$, co-prime to $n$, for which the rank-$1$ lattice $P_n=\{(\{kp/n\},k/n),0\leq k<n\}\subseteq[0,1]^2$ satisfies
\[
\frac{C}{n}>\disp_\mathbb{T}(P_n)=\frac{1}{n}\max_{0<i<k}\max_{0\leq j<a_i}\frac{\vol(B_{A_i+j})}{\det(\Lambda)},
%\big(=\frac{p}{n^2}\max_{0<i<k}\max_{0\leq j<a_i}\vol(B_{A_i+j})\big)
\]
where the lattice $\Lambda$ is defined as in the statement of Lemma \ref{lem:FindTorusPointSet}.
Applying the lower bound of Theorem \ref{thm:Genreal upper/lower bounds} we see that the continued fraction coefficients of $p/n=[a_1,\dots,a_k,1]$ must have
\[
C>\frac{a_i}{4}+1+\frac{1}{a_i}-\frac{1}{4a_i}>\frac{a_i}{4},
\]
which gives $a_i<4C$. We conclude that Zaremba's Conjecture holds with constant $A=4C$.
\end{proof}

 %%%%%%%%%%%%%%%%%%%%%%%%%%%%%%%%%%%%%%%%%%%%%%%%%%%%%%%%%%%%%%%%%%%%%%%%%%%%%%%%%%%%%%%%%%
\section{Conclusions and future work}

Because our results rely so heavily on the continued fraction expansion of the generators of a lattice and since most numbers do not have a well-behaved expansion, finding the exact dispersion of a lattice is hard. The easiest to handle are quadratic lattices, where the dispersion depends only on the discriminant of the underlying ring (or of the determinant of the lattice). Fortunately, the most of the interesting lattices are quadratic lattices. After that, the dispersion of lattices whose associated sequence is periodic can be calculated explicitly with the help of a computer and the lattices with lowest dispersion are of this type. Another case we can handle are lattices generated by a number known to have unbounded continued fraction coefficients, in which case the dispersion is infinite. Finally, if we happen to know what the largest continued fraction coefficient of the generators are we can estimate the dispersion of the lattice. 

Theorem \ref{thm:Zaremba eq disp Bound} connects periodic dispersion to a long-standing conjecture. The methods used in \cite{BreneisHinrichs} do not rely on continued fractions to compute the dispersion of an integral lattice, instead they use what they call a splitting argument.  It is possible that dispersion could be connected to other conjectures. 

We would also like to investigate the connection to the Markov and Lagrange spectrum. Interestingly, in \cite{niederreiter1984measure} the author considered the dispersion of 1-dimensional sequences and showed that sequences associated with $\varphi$ and $\sqrt{2}$ give the best possible constants. The author of \cite{OndispersionandMarkovconstants} was able to show that a larger associated Markov constant does not automatically imply larger dispersion, as one might expect by the findings in \cite{niederreiter1984measure}. Furthermore, in \cite{tripathi1993comparison} a bigger family of pairs that share this property. Note that the numbers of best possible dispersion the author found therein do not match the values we found in section \ref{thm:BestLattices}.

Well-known objects in Algebraic number theory seem to produce very good lattices with respect to dispersion. The ring of integers $R$ of a real number field $K$ is associated to the lattice
\[
\Lambda=\{(\sigma_1(\alpha),\dots,\sigma_s(\alpha)):\alpha\in R\},
\]
where $\sigma_1,\dots,\sigma_s$ be the $s$ real embeddings of a real number field. For these lattices the volumes of a maximal empty axis-parallel box $B\subset\Lambda$ take finitely many values. This is because boxes can be transformed by scaling each coordinate by an embedding of a given unit without changing either its volume or how many lattice points it contains. By reducing maximal empty boxes, for lattices with small discriminant in moderately high dimensions it should be possible to develop a relatively efficient algorithm that computes actual dispersion of the lattice.  In particular it would be interesting to verify whether or not the lattices in \cite[Table 1]{kacwin2018numerical} are optimal.

We would like to obtain good bounds for the dispersion of such lattices in terms of the discriminant of the underlying number ring. However, several things in the two-dimensional case do not generalize to higher dimensions. For one, in two-dimensions any box is completely determined by two opposing points, while in higher dimensions you could need up to $d-1$ opposing points. Continued fractions also do not generalize to higher dimensions, so if our methods were to be adapted one would need to work with a higher dimensional generalization of continued fractions. Also, it is not clear whether or not it is possible to translate a maximal empty axis-parallel box so that it is strictly contained in another empty box.

We have used the dispersion package developed by Benjamin Sommer, which will be available at some point later, to estimate the dispersion of the lattices associated with the rings of integers of cubic fields with small discriminant. There is a definite correlation between the discriminant and the dispersion of the lattice, but the authors cannot tell if it is linear as in the quadratic case. The cubic lattice with lowest discriminant is associated to the ring of integers of the splitting field $\Q[2\cos(\frac{\pi}{7})]$ of the polynomial $x^3-x^2-2x+1$, which has generating matrix
\[
2\begin{pmatrix}
\cos(\frac{\pi}{7}) &   \cos(\frac{3\pi}{7}) &   \cos(\frac{5\pi}{7})\\
\cos(\frac{3\pi}{7}) &   \cos(\frac{5\pi}{7}) &   \cos(\frac{\pi}{7})\\
\cos(\frac{5\pi}{7}) &   \cos(\frac{\pi}{7}) &   \cos(\frac{3\pi}{7})
\end{pmatrix}.
\]
The only possible values for the normalized volume of maximal empty axis-parallel boxes amidst this lattice are roughly $2.74224$, $2.92038$, and $3.38423$. This means that, for this lattice, the normalized dispersion is $3.38423\dots$, which we conjecture to be the best possible dispersion for three-dimensional lattices.

The optimal two-dimensional and the conjectured optimal three-dimensional lattice are associate to the ring of integers of the real number field $\Q[2\cos(\pi/5)]$ and $\Q[2\cos(\pi/7)]$. These are special cases of the real number fields $\Q[2\cos(\pi/(2d+1)]$ where $2d+1$ is prime and $d$ is the dimension of the lattice. We have also verified that, for small $d$, this field has the smallest discriminant when $d$ is also prime (a prime $d$ such that $2d+1$ is also prime is called a Sophie Germain prime).  The authors conjecture that the dispersion of these lattices will have the optimal dependence on $d$. It may be possible to compute the dispersion of these fields directly for all $d$.

\printbibliography

@misc{Wiart,
      title={Improved dispersion bounds for modified Fibonacci lattices}, 
      author={Ralph Kritzinger and Jaspar Wiart},
      year={2020},
      eprint={2007.02297},
      archivePrefix={arXiv},
      primaryClass={math.CO}
}

@misc{Bukh,
      title={Empty axis-parallel boxes}, 
      author={Boris Bukh and Ting-Wei Chao},
      year={2020},
      eprint={2009.05820},
      archivePrefix={arXiv},
      primaryClass={math.CO}
}

@inbook{Larcher,
author = {Gerhard Larcher},
editor = {Peter Kritzer and Harald Niederreiter and Friedrich Pillichshammer and Arne Winterhof},
doi = {doi:10.1515/9783110317930.171},
url = {https://doi.org/10.1515/9783110317930.171},
title = {Discrepancy estimates for sequences: new results and open problems},
booktitle = {Uniform Distribution and Quasi-Monte Carlo Methods},
year = {2014},
publisher = {De Gruyter},
pages = {171--190}
}

@article{Ullrich,
title = {A lower bound for the dispersion on the torus},
journal = {Mathematics and Computers in Simulation},
volume = {143},
pages = {186-190},
year = {2018},
note = {10th IMACS Seminar on Monte Carlo Methods},
issn = {0378-4754},
doi = {https://doi.org/10.1016/j.matcom.2015.12.005},
url = {https://www.sciencedirect.com/science/article/pii/S0378475415002736},
author = {Mario Ullrich}
}

@inbook{BreneisHinrichs,
author = {Simon Breneis and Aicke Hinrichs},
editor = {Dmitriy Bilyk and Josef Dick and Friedrich Pillichshammer},
doi = {doi:10.1515/9783110652581-006},
url = {https://doi.org/10.1515/9783110652581-006},
title = {6. Fibonacci lattices have minimal dispersion on the two-dimensional torus},
booktitle = {Discrepancy Theory},
year = {2020},
publisher = {De Gruyter},
pages = {117--132}
}

@misc{kacwin2018numerical,
      title={Numerical performance of optimized Frolov lattices in tensor product reproducing kernel Sobolev spaces}, 
      author={Christopher Kacwin and Jens Oettershagen and Mario Ullrich and Tino Ullrich},
      year={2018},
      eprint={1802.08666},
      archivePrefix={arXiv},
      primaryClass={math.NA}
}

@article{Fiat1989HowTF,
  title={How to find a battleship},
  author={A. Fiat and A. Shamir},
  journal={Networks},
  year={1989},
  volume={19},
  pages={361-371}
}

@article {OndispersionandMarkovconstants,
      author = "V. Drobot",
      title = "On dispersion and Markov constants",
      journal = "Acta Mathematica Hungarica",
      year = "1986",
      publisher = "Akadémiai Kiadó, co-published with Springer Science+Business Media B.V., Formerly Kluwer Academic Publishers B.V.",
      address = "Budapest, Hungary",
      volume = "47",
      number = "1-2",
      doi = "10.1007/bf01949128",
      pages=      "89 - 93",
      url = "https://akjournals.com/view/journals/10473/47/1-2/article-p89.xml"
}

@article{niederreiter1984measure,
  title={On a measure of denseness for sequences},
  author={Niederreiter, H},
  journal={Topics in classical number theory},
  volume={2},
  pages={1163--1208},
  year={1984},
  publisher={North Holland Amsterdam}
}

@article{tripathi1993comparison,
  title={A comparison of dispersion and Markov constants},
  author={Tripathi, Amitabha},
  journal={Acta Arithmetica},
  volume={63},
  pages={193--203},
  year={1993},
  publisher={Instytut Matematyczny Polskiej Akademii Nauk}
}

@misc{kritzinger2020dispersion,
      title={Dispersion of digital $(0,m,2)$-nets}, 
      author={Ralph Kritzinger},
      year={2020},
      eprint={2004.14760},
      archivePrefix={arXiv},
      primaryClass={math.MG}
}

@article{KRIEG2018115,
title = {On the dispersion of sparse grids},
journal = {Journal of Complexity},
volume = {45},
pages = {115-119},
year = {2018},
issn = {0885-064X},
doi = {https://doi.org/10.1016/j.jco.2017.11.005},
url = {https://www.sciencedirect.com/science/article/pii/S0885064X17301061},
author = {David Krieg}
}

@article{ULLRICH2019385,
title = {A note on the dispersion of admissible lattices},
journal = {Discrete Applied Mathematics},
volume = {257},
pages = {385-387},
year = {2019},
issn = {0166-218X},
doi = {https://doi.org/10.1016/j.dam.2018.08.032},
url = {https://www.sciencedirect.com/science/article/pii/S0166218X18305055},
author = {Mario Ullrich}
}
\clearpage
\appendix
\section{Proof of Proposition \ref{thm:EstimateOnAi}}\label{sec:Proof of EstimateOnAi}

% \begin{defi}\jaspar{I have defined the fundamental unit to be greater than 1.}\thomas{that is fine. This way it becomes unique (which is not too important for me). I will check in other parts to write it in a way that is according to this definition.}
% The \emph{fundamental unit} $\varepsilon\in\Z[\delta]$ is the unique unit that satisfies
% \begin{enumerate}
%     \item Every unit in $\Z[\delta]$ is of the form $\pm\varepsilon^k$ for some $k\in\Z$,
%     \item $\varepsilon>1$.
% \end{enumerate}
% \end{defi}

% The following proposition is well-known in the case of the ring of integers of a quadratic field and is proved in Appendix B.

\begin{lemma}\label{lem:OnlyPeriodUnit}
For $\Delta = [\overline{a_0,a_1,\dots, a_{l-1}}]$ being a purely periodic quadratic integers, corresponding to the polynomial $x^2-tx+n$, we have $(-1)^i N(p_i - q_i \Delta)=1 $ if and only if $i\in l \Z$. In addition, if $\Delta$ is a purely periodic quadratic number the if part still holds.
\end{lemma}

\begin{proof}
Let $\Delta$ be a purely periodic quadratic number. We begin by rewriting $\Delta_i$ with the help of Lemma \ref{lem: Delta i} as
\begin{align*}
    \Delta_i& =\frac{-\Delta q_{i-1} + p_{i-1}}{\Delta q_i - p_i} = \frac{-\Delta q_{i-1} + p_{i-1}}{\Delta q_i - p_i} \cdot \frac{\overline{\Delta} q_i - p_i}{\overline{\Delta} q_i - p_i}  \\
    &= \frac{(-p_ip_{i-1}-\Delta\overline{\Delta}q_iq_{i-1})+p_{i-1}q_i\overline{\Delta}+p_{i}q_{i-1}\Delta}{N(p_i - \Delta q_i)} + \frac{p_{i-1}q_i\Delta-p_{i-1}q_i\Delta}{N(p_i - \Delta q_i)}\\
    &=\frac{(-p_ip_{i-1}-nq_iq_{i-1}+tp_{i-1}q_i)+(-p_{i-1}q_i+p_{i}q_{i-1})\Delta}{N(p_i - \Delta q_i)}\\
    &=\frac{(-p_ip_{i-1}-nq_iq_{i-1}+tp_{i-1}q_i)+(-1)^i\Delta}{N(p_i - \Delta q_i)}\\
    &=\frac{(-1)^{i+1}(p_ip_{i-1}+nq_iq_{i-1}-tp_{i-1}q_i)+\Delta}{(-1)^i N(p_i - \Delta q_i)}\\
    &=\frac{m_i+\Delta}{(-1)^i N(p_i - \Delta q_i)}.
\end{align*}
If $i\in l \Z$, then it is $\Delta_i=\Delta$ and thus
\[
\frac{m_i+\Delta}{(-1)^i N(p_i - \Delta q_i)}=\Delta,
\]
or equivalently
\[
m_i+\Delta=(-1)^i N(p_i - \Delta q_i)\Delta.
\]
Comparing coefficients it follows that $m_i=0$ and $(-1)^i N(p_i - \Delta q_i)=1$.

Assume in addition that $\Delta$ is a quadratic integer, then we certainly have $m_i=(-1)^{i+1}(p_ip_{i-1}+nq_iq_{i-1}-tp_{i-1}q_i)\in \Z$.
Assume now that $(-1)^i N(p_i - q_i \Delta)=1 $, meaning $\Delta_i = m_i + \Delta$. Both, $\Delta$ and $\Delta_i$ are purely periodic which implies $-1<\overline{\Delta}<0$ and $-1<m_i+\overline{\Delta}<0$. Since $m_i\in\Z$ this can only be the case if $m_i=0$, showing $\Delta_i=\Delta$ which is only the case if $i\in\l\Z$.
\end{proof}

% and therefore
% \[
% |N(p_{i}-q_{i}\Delta)|=|N(p_{i+l}-q_{i+l}\Delta)|.
% \]

\begin{lemma}
Let $\Delta = [\overline{a_0,a_1,\dots, a_{l-1}}]$ be a purely periodic quadratic number. Then for all $i\in\Z$
\begin{align}\label{eq:PeriodicNormEq}
|N(p_{i}-q_{i}\Delta)|=|N(p_{i+kl}-q_{i+kl}\Delta)|.
\end{align}
\end{lemma}

\begin{proof}
This follows directly from
\begin{align}
(p_{i}-q_{i}\Delta)(p_{l}-q_{l}\Delta) &= p_{i+l}-q_{i+l}\Delta \quad \text{and}\label{eq:PeriodicMult}\\
(p_{i}-q_{i}\Delta)(p_{l}-q_{l}\overline{\Delta}) &= (-1)^l(p_{i-l}-q_{i-l}\Delta).\label{eq:PeriodicMultBar}
\end{align}
Equality \eqref{eq:PeriodicMult} is obviously true for $i=0$. Noting that $a_{i}=a_{i+l}$ and calculating
\begin{align*}
   &\phantom{=} (p_{i+1}-q_{i+1}\Delta)(p_{l}-q_{l}\Delta) \\
   &= (a_{i}p_{i}+p_{i-1}-(a_{i}q_{i}+q_{i-1})\Delta)(p_{l}-q_{l}\Delta) \\
   &= (a_{i}(p_{i}-q_{i}\Delta)+(p_{i-1}-q_{i-1}\Delta))(p_{l}-q_{l}\Delta) \\
   &= a_{i+l}(p_{i+l}-q_{i+l}\Delta)+p_{i+l-1}-q_{i+l-1}\Delta \\
   &= (a_{i+l}p_{i+l}+p_{i+l-1})-(a_{i+l}q_{i+l}+q_{i+l-1})\Delta \\
   &= p_{i+1+l}-q_{i+1+l}\Delta,
\end{align*}
we see by induction that \eqref{eq:PeriodicMult} holds for $i\in \N_0$
Equality (\ref{eq:PeriodicMultBar}) follows similarly with the help of the identity 
\begin{align}\label{eq:Conjugate}
   (p_{i}-q_{i}\overline{\Delta}) & = (\overline{p}_{-i}-\overline{q}_{-i}\overline{\Delta}) = (-1)^i (p_{-i} - a_0 q_{-i} + q_{-i} (a_0 -\Delta))= (-1)^i (p_{-i} - q_{-i} \Delta).
\end{align}
So, we have that
\[
(p_{l}-q_{l}\Delta)^k=(p_{kl}-q_{kl}\Delta)
\]
and together with equation (\ref{eq:PeriodicMultBar}) arrive at the assertion.
\end{proof}

\begin{lemma}\label{lem:NumberThoeryEstimate1}
For $i\in\Z$, let $p_i/q_i$ be the $i$-th convergent of the purely periodic number $\Delta$. Then
\[
\frac{\Delta-\overline{\Delta}}{a_{i}+2}\leq|N(p_{i}-q_{i}\Delta)|\leq \frac{\Delta-\overline{\Delta}}{a_{i}}
\]
\end{lemma}

\begin{proof}

We begin by calculating
\begin{align*}
|N(p_{i}-q_{i}\Delta)|&=|(p_{i}-q_{i}\Delta)(p_{i}-q_{i}\overline{\Delta})| \\
&=|(p_{i}-q_{i}\Delta)(p_{i}-q_{i}\big(\overline{\Delta}-\Delta+\Delta)\big)|\\
&=|(p_{i}-q_{i}\Delta)((p_{i}-q_{i}\Delta)-q_{i}(\overline{\Delta}-\Delta)|\\
&=|(p_{i}-q_{i}\Delta)^2-q_{i}(p_{i}-q_{i}\Delta)(\overline{\Delta}-\Delta)|\\
&=|(p_{i}-q_{i}\Delta)^2+q_{i}(p_{i}-q_{i}\Delta)(\Delta-\overline{\Delta})|
\end{align*}
and in the same way
\begin{align*}
|N(p_{i}-q_{i}\Delta)|&= |(p_{i}-q_{i}\overline{\Delta})^2+q_{i}(p_{i}-q_{i}\overline{\Delta})(\overline{\Delta}-\Delta)|,
\end{align*}
we need both equalities to be able to deal with both positive and negative values of $i$.

With the previous lemma and \eqref{eq:Convergents Fact 4}, as well as $a_i=a_{i+kl}$ we get
\[
|N(p_{i}-q_{i}\Delta)|=|N(p_{i+kl}-q_{i+kl}\Delta)|\leq \frac{1}{a_{i}^2q_{i+kl}^2}+\frac{\Delta-\overline{\Delta}}{a_{i}}.
\]
Since this is true for arbitrary large $k\in\Z$, we obtain the upper bound
\[
|N(p_{i}-q_{i}\Delta)| \leq \frac{\Delta-\overline{\Delta}}{a_{i}}.
\]
Using the reverse triangle inequality we get the lower bound
\begin{align*}
|N(p_{i}-q_{i}\Delta)|&=|(p_{i}-q_{i}\Delta)^2+q_{i}(p_{i}-q_{i}\Delta)(\Delta-\overline{\Delta})| \ge |q_{i}(p_{i}-q_{i}\Delta)(\Delta-\overline{\Delta})| - |(p_{i}-q_{i}\Delta)^2|
\end{align*}
and 
\begin{align*}
|N(p_{i}-q_{i}\Delta)|&= |(p_{i}-q_{i}\overline{\Delta})^2+q_{i}(p_{i}-q_{i}\overline{\Delta})(\overline{\Delta}-\Delta)|\ge |q_{i}(p_{i}-q_{i}\overline{\Delta})(\Delta-\overline{\Delta})| - |(p_{i}-q_{i}\overline{\Delta})^2|.
\end{align*}
With this we get
\[
|N(p_{i}-q_{i}\Delta)|=|N(p_{i+kl}-q_{i+kl}\Delta)|\ge \frac{\Delta-\overline{\Delta}}{a_{i}+2} - \frac{\Delta-\overline{\Delta}}{a_{i}^2q_{i+kl}^2},
\]
which, again, holds for arbitrarily large $k\in \mathbb{Z}$ and we arrive at
\[
|N(p_{i}-q_{i}\Delta)|\ge \frac{\Delta-\overline{\Delta}}{a_{i}+2}. \qedhere
\]
\end{proof}
With this result used for purely periodic quadratic integers, we can finally give a proof of Theorem \ref{thm:EstimateOnAi}.
\begin{proof}[Proof of Theorem \ref{thm:EstimateOnAi}]

Assume to the contrary that $a_{i}>\frac{\Delta - \overline\Delta}{2}$.
By the previous lemma we obtain
\[
|N(p_{i}-q_{i}\Delta)|\leq \frac{\Delta-\overline{\Delta}}{a_{i}}
<2.
\]
Since the norm of a quadratic integer has to be an integer value we have $|N(p_{i}-q_{i}\Delta)|=1$, but as seen in Lemma \ref{lem:OnlyPeriodUnit} this can only be the case if $i \in l\Z$. This leads to a contradiction to the choice of $0<i<l$. The rest of the statement follows from $a_0=\Delta+\overline{\Delta}$.
\end{proof}

As corollary we get the assertion about $\sqrt{d}$.

\begin{cor}
The continued fraction expansion of $\sqrt{d}=[a_0,\overline{a_1,\dots,a_{l-1},2a_0}]$ satisfies $a_i\leq a_0$ for all  $0<i<l$.
\end{cor}

\begin{proof}
It is well known that the continued fraction expansion of $\sqrt{d}$ is given by $[a_0,\overline{a_1,\dots,a_{l-1},2a_0}]$. The remainder follows directly from $\lfloor\sqrt{d} \rfloor + \sqrt{d} = [2a_0,\overline{a_1,\dots,a_{l-1},2a_0}]$ is a purely periodic quadratic integer.
\end{proof}

We want to state a conjecture on all coefficients of such $\Delta$.

\begin{conjecture}
Let $\Delta=[\overline{a_0,a_1,\dots,a_{l-1}}]$ be a purely periodic quadratic integer and let $\mathcal{C}(\Delta)=\{a_0,a_1,\dots,a_{l-1}\}$.
Then the $n$-th largest element $w_n$ of $\mathcal{C}(\Delta)$ satisfies
\[
w_n\leq\frac{a_0}{n}.
\]
\end{conjecture}

\clearpage
\section{Figures}
   \begin{figure}[bp!]
\captionsetup[subfigure]{labelformat=empty}
     \centering
     \begin{subfigure}[b]{0.475\textwidth}
         \centering
         \includegraphics{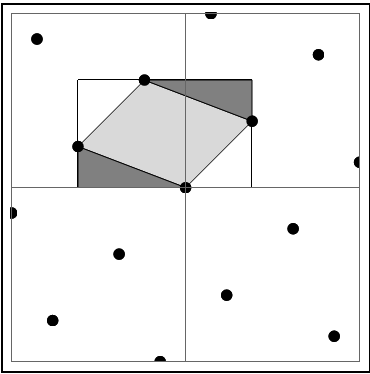}
         \caption{$B_0\square\Lambda_5$}
         \label{fig:Polygon 5}
     \end{subfigure}
     \hspace{.025\textwidth}
     \begin{subfigure}[b]{0.475\textwidth}
         \centering
         \includegraphics{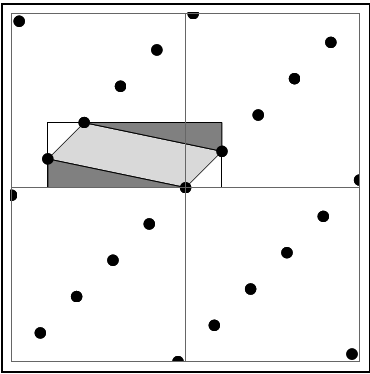}
         \caption{$B_0\square\Lambda_{21}$}
         \label{fig:Polygon 21}
     \end{subfigure}
\caption{The box $B_0$ with the fundamental, the parallelepiped $P$ generated by the basis $\{(1,1),(-\Delta,-\overline\Delta)\}$ shown in light gray, and the lower left corner, $L$, and upper right corner, $R$, shown in dark gray.}
\label{fig:Polygons}
\end{figure}
\begin{figure}[bp!]
\captionsetup[subfigure]{labelformat=empty}
     \centering
     \begin{subfigure}[b]{0.475\textwidth}
         \centering
         \includegraphics{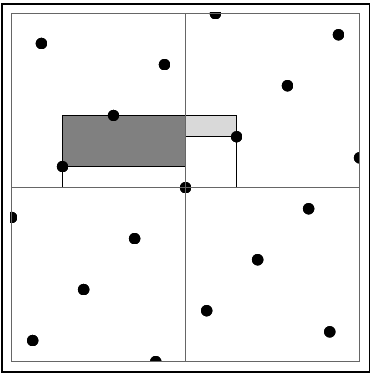}
         \caption{$B_0\square\Lambda_2$}
         %\caption{The white boxes to the left and right of the $y$-axis are $B(\alpha_0)$ and $B(\beta_0)$ respectively, the dark gray box is $C_{0,1}$, and the light gray box is $C_{0,2}$.}
         \label{fig:Rearange 0}
     \end{subfigure}
     \hspace{.025\textwidth}
     \begin{subfigure}[b]{0.457\textwidth}
         \centering
         \includegraphics{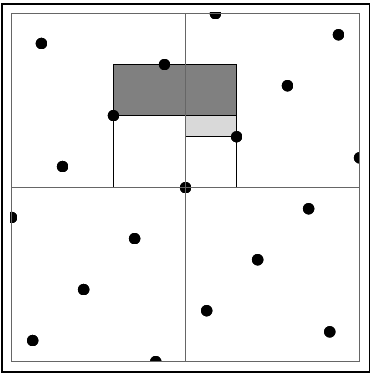}
         \caption{$B_1\square\Lambda_2$}
        % \caption{The white boxes to the left and right of the $y$-axis are $B(\alpha_1)$ and $B(\beta_1)$ respectively, the dark gray box is $D_{1,1}$ and the light gray box is $D_{1,2}$.} 
         \label{fig:Rearange 1}
     \end{subfigure}
\caption{The partitions in the proof of Lemma \ref{prop:VolumeIsNormsPlusConstant}. The white boxes to the left and right of the $y$-axis are $B(\alpha_n)$ and $B(\beta_n)$ respectively. The dark gray boxes inside $B_0$ and $B_1$ are $C_{0,1}$ and $D_{1,1}$ respectively; these regions are translates. The light gray boxes inside $B_0$ and $B_1$ are $C_{0,2}$ and $D_{0,2}$ respectively; these regions are the same. The area of $B(\alpha_n)$ is $|\alpha_n\widetilde\alpha_n|$, the area of $B(\beta_n)$ is $|\beta_n\widetilde\beta_n|$, and the area of gray regions within $B_0$ and $B_1$ is equal to $\Delta-\overline\Delta$.}
\label{fig:Rearange Boxes}
\end{figure}
\begin{figure}[bp!]
\captionsetup[subfigure]{labelformat=empty}
     \centering
     \begin{subfigure}[b]{0.475\textwidth}
         \centering
         \includegraphics{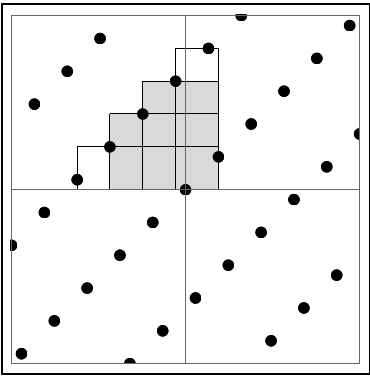}
         \caption{$\Delta = 1+\delta_{13}=[\overline3],\quad (\Delta-\overline\Delta)^2=13$}
         %\caption{The white boxes to the left and right of the $y$-axis are $B(\alpha_0)$ and $B(\beta_0)$ respectively, the dark gray box is $C_{0,1}$, and the light gray box is $C_{0,2}$.}
         \label{fig:Rearange 0}
     \end{subfigure}
     \hspace{.025\textwidth}
     \begin{subfigure}[b]{0.457\textwidth}
         \centering
         \includegraphics{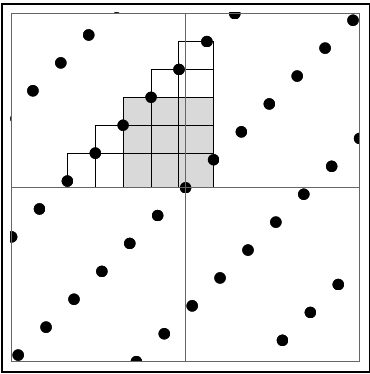}
         \caption{$\Delta=1+2\delta_5=[\overline4] ,\quad (\Delta-\overline\Delta)^2=20$}
        % \caption{The white boxes to the left and right of the $y$-axis are $B(\alpha_1)$ and $B(\beta_1)$ respectively, the dark gray box is $D_{1,1}$ and the light gray box is $D_{1,2}$.} 
         \label{fig:Rearange 1}
     \end{subfigure}
\caption{The boxes corresponding to the first period of $\Delta$ for two different lattices with the largest box(es) in gray. Notice that number of largest boxes depends on the parity of  $(\Delta-\overline\Delta)^2$. See Figure \ref{fig:Cases} for the continued fraction expansion of the generators of the lattices.}
\label{fig:Biggest Boxes}
\end{figure}
\begin{figure}[h]  
     \centering
     {\includegraphics{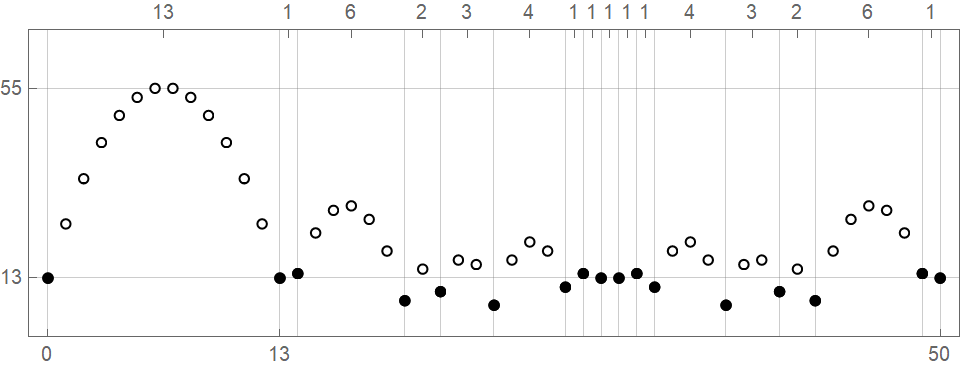}}
     \caption{A plot of $(n,|N(\alpha_n)|+|N(\beta_n)|)$ for $\Delta=6+\delta_{217}=[\overline{13,1,6,2,3,4,1,1,1,1,1,4,3,2,6,1}]$ and $0\leq n\leq 50$. Lemma \ref{prop:VolumeIsNormsPlusConstant} together with the fact that $\Delta$ is purely periodic tells us that the area of any maximal empty axis-parallel box $B$ amidst the points of the lattice generated by $\{(1,1),(-\Delta,-\overline\Delta)\}$ is equal to one of $|N(\alpha_n)|+|N(\beta_n)|+\Delta-\overline\Delta$, where $0\leq n\leq 50$. The height of black points represent the norm part of the volume of the $B_n$ when $n=A_i$, the left and right points bounding these boxes correspond to convergents of $\Delta$. The hollow points appear on parabolas between the black points (one of $\alpha_n$ or $\beta_n$ corresponds to a semiconvergent of $\Delta$). Theorem \ref{thm:Genreal upper/lower bounds} estimates the height of the parabolas based on the gaps between the black points, i.e. the continued fraction coefficients. An upper bound on the length of gaps between the black points is calculated in Theorem \ref{thm:EstimateOnAi}.}
     \label{fig:Discriminant 217}
\end{figure}

\clearpage
\begin{figure}
\captionsetup[subfigure]{labelformat=empty}
     \centering
     \begin{subfigure}[b]{0.32\textwidth}
         \centering
         \includegraphics{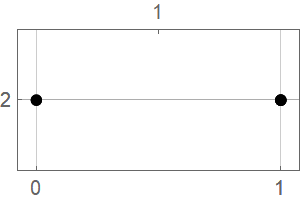}
         \caption{$\Delta=\delta_{5},\quad (\Delta-\overline\Delta)^2=5$}
         \label{fig:Discriminant 5}
     \end{subfigure}
     \begin{subfigure}[b]{0.32\textwidth}
         \centering
         \includegraphics{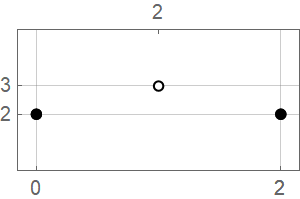}
         \caption{$\Delta=1+\delta_{2},\quad (\Delta-\overline\Delta)^2=8$}
         \label{fig:Discriminant 8}
     \end{subfigure}
     \begin{subfigure}[b]{0.32\textwidth}
         \centering
         \includegraphics{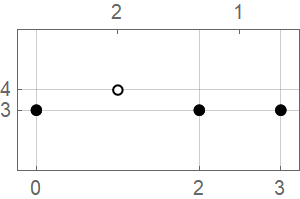}
         \caption{$\Delta=1+\delta_{3},\quad (\Delta-\overline\Delta)^2=12$}
         \label{fig:Discriminant 12}
     \end{subfigure}\\
% \end{figure}\vspace{2mm}
% \begin{figure}[bp!]\captionsetup[subfigure]{labelformat=empty}\centering
       \begin{subfigure}[b]{0.32\textwidth}
         \centering
         \includegraphics{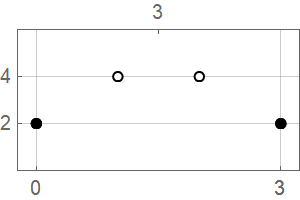}
         \caption{$\Delta=1+\delta_{13},\quad (\Delta-\overline\Delta)^2=13$}
         \label{fig:Discriminant 13}
     \end{subfigure}
     \begin{subfigure}[b]{0.32\textwidth}
         \centering
         \includegraphics{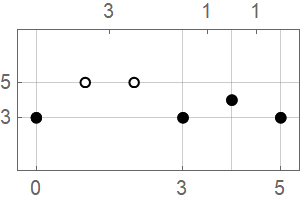}
         \caption{$\Delta=1+\delta_{17},\quad (\Delta-\overline\Delta)^2=17$}
         \label{fig:Discriminant 17}
     \end{subfigure}
     \begin{subfigure}[b]{0.32\textwidth}
         \centering
         \includegraphics{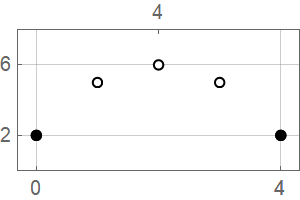}
         \caption{$\Delta=1+2\delta_{5},\quad (\Delta-\overline\Delta)^2=20$}
         \label{fig:Discriminant 20}
     \end{subfigure}
     \caption{The explicit cases of Theorem \ref{thm:MainTheorem} that need to be checked. Notice that the highest dot appears between the first two black dots.}
\label{fig:Cases}
\end{figure}\vspace{2mm}
\begin{table}
\begin{center}
\begin{tabular}{|c|c|c|c|c|}
\hline
    $\max_{i\in\Z}a_i$ & $L(a_i)$ & $\disp(\overline a)$ & $\disp(\overline{a,1})$ & $U(a_i)$ \\
\hline
    2 & 2 & 2.06066 & 2.1547 & 2.25 \\
\hline
    3  & 2 & 2.1094 & 2.30931 & 2.4 \\
    \hline
     4 & 2.25 & 2.34164 & 2.59099  & 2.66667 \\
      \hline
     5 & 2.4 & 2.48556 & 2.78885  & 2.85714 \\
      \hline
     6 & 2.66667 & 2.73925 & 3.06559 & 3.125 \\
      \hline
      7& 2.85714 & 2.92305 & 3.27921 & 3.33333 \\
      \hline
     8 & 3.125 &3.18282  & 3.55155 & 3.6 \\
      \hline
     9 & 3.33333 & 3.38624 & 3.7735 & 3.81818 \\
      \hline
     10 & 3.6 & 3.64757 & 4.04256 & 4.08333 \\
     \hline
      25 & 7.28 & 7.29987  & 7.75931 & 7.77778 \\
      \hline
       50  & 13.52 & 13.53  & 14.0096 & 14.0192 \\
      \hline
       75  & 19.76 & 19.7667 & 20.2532 & 20.2597 \\
      \hline
       100  & 26.01 & 26.015 & 26.5049 & 26.5098 \\
      \hline
        150       & 38.5067 & 38.51 & 39.0033 & 39.0066 \\
      \hline
          200     & 51.005 & 51.0075 & 51.5025 & 51.505 \\
      \hline
          500     & 126.002 & 126.003 & 126.501 & 126.502 \\
      \hline
        1000     & 251.001 & 251.001 & 251.5 & 251.501 \\
      \hline
\end{tabular}
\caption{The lower bounds and upper bounds from Theorem \ref{thm:Genreal upper/lower bounds} and Theorem \ref{thm:Thight Bound} rounded to 5 decimal places.}{\label{tab:dispbyai}}
\end{center}
\end{table}

\end{document}